\documentclass[a4paper,12pt]{article}

\usepackage[a4paper,top=3cm,bottom=2cm,left=3cm,right=3cm,marginparwidth=1.75cm]{geometry}

\usepackage{amsmath}
\usepackage{dsfont}
\usepackage{graphicx}
\usepackage{subfigure}
\usepackage[colorinlistoftodos]{todonotes}
\usepackage[colorlinks=true, allcolors=blue]{hyperref}
\usepackage{amsthm}
\usepackage{float}
\usepackage{hyperref}
\usepackage{float}
\theoremstyle{plain}
\newtheorem{theorem}{Theorem}[section]
\newtheorem{lemma}{Lemma}[section]
\newtheorem{remark}[theorem]{Remark}

\numberwithin{equation}{section}
\numberwithin{lemma}{section}

\title{Extended Sampling Method in Inverse Scattering}
\author{Juan Liu\footnotemark[1] and Jiguang Sun\footnotemark[2]}
\date{}
\begin{document}
\footnotetext[1]{Department of Mathematical Sciences, Jinan University, Guangzhou, 130012, China.
E-mail: \texttt{liujuan@jnu.edu.cn}}
\footnotetext[2]{Department of Mathematical Sciences, Michigan Technological University, Houghton, MI 49931, U.S.A.
and School of Mathematical Sciences, University of Electronic Science and Technology of China, Chengdu, 611731, China.
E-mail: \texttt{jiguangs@mtu.edu}}
\maketitle
\begin{abstract}
A new sampling method for inverse scattering problems is proposed to process far field data of one incident wave.
As the linear sampling method, the method sets up ill-posed integral equations and uses the (approximate) solutions to reconstruct the target. 
In contrast, the kernels of the associated integral operators are the far field patterns of sound soft balls.
The measured data is moved to right hand sides of the equations, which gives the method the ability to process limit aperture data.
Furthermore, a multilevel technique is employed to improve the reconstruction. 
Numerical examples show that the method can effectively determine the location and approximate the support with little a priori information of the unknown target.  
\end{abstract}

\section{Introduction}
The inverse scattering theory has been an active research area for more than thirty years and is still developing \cite{ColtonKress2009IPI}.
Various sampling methods have been proposed to reconstruct the location and support of the unknown scatterer, e.g., the linear sampling method, the factorization method, the reciprocity gap method \cite{ColtonKirsch1996IP, Kirsch1998IP, KirschGrinberg2008, Potthast2006IP, ColtonHaddar2005IP, DiCristoSun2006IP, MonkSun2007IPI, ColtonKress2013, CakoniColton2014}. Using the scattering data (far field pattern or near field data), these methods solve some linear ill-posed integral equations for each point in the sampling domain containing the target. The regularized solutions are used as indicators for the sampling points and the support of the scatterer is reconstructed accordingly.

Classical sampling methods use full aperture data to set up the linear ill-posed integral equations. For example, the linear sampling method (LSM) uses the far field patterns of all scattering directions for plane incident waves of all directions. In this paper, we propose a new method, the extended sampling method (ESM), to reconstruct the location and approximate the support of the scatterer using much less data. Similar to the LSM, at each sampling point in a domain containing the target, the ESM solves a linear ill-posed integral equation. However, the kernel of the integral operator is the far field pattern of a sound soft ball with known center and radius rather than the measured far field pattern of the unknown scatterer. The measured far field pattern is moved to the right hand side of the integral equation. As a consequence, the ESM can treat limit aperture data naturally. 
Furthermore, the ESM is independent of the wave numbers, i.e., the ESM works for all wavenumbers. In contrast, the LSM needs to exclude wavenumbers which are certain eigenvalues related to the scatterer.

In recent years, some direct methods were proposed to reconstruct the scatterer using far field pattern of one incident wave, see, e.g., 
\cite{PotthastSylvestrKusiak2003IP, Potthast2010IP, ItoJinZou2012IP, Liu2017IP, LiuSun2018,LiuDai2015JCAM}.
The ESM is different from these methods in the sense that it is based on the classical sampling method. 
Furthermore, the behavior of the solutions of the linear ill-posed integral equations can be theoretically justified. 

The rest of the paper is arranged as follows. In Section \ref{SPFFP}, we introduce the scattering problems and the far field pattern. In addition, we give the far field pattern of the scattered field due to sound soft balls, which serves as the kernel the integral operator. In Section \ref{ESM}, the ESM based on a new far field equation is presented. The behavior of the solutions is analyzed.  In Section \ref{NEs}, numerical examples are provided to show the effectiveness of the proposed method. Finally, we draw some conclusions and discuss future works in Section \ref{CF}. For simplicity, the presentation is restricted to 2D. The extension to 3D is straightforward.

\section{Scattering Problems and Far Field Pattern}\label{SPFFP}
We consider the acoustic obstacle and medium scattering problems in $\mathds{R}^2$. 
Let $D\subset \mathds{R}^2$ be a bounded domain with a $C^2$ boundary. We denote by $u^{i}(x):={\rm e}^{i kx\cdot d}, x, d \in \mathds{R}^2$ the incident plane 
wave, where $|d|=1$ is the direction and $k>0$ is the wavenumber. The scattering problem is to find the scattered field $u^s$ or the total field $u=u^i+u^s$ such that
\begin{equation}\label{direct1}
  \left\{
   \begin{array}{ll}
   &\Delta u+k^2 u=0,\ \ \ \textrm{in}\  \mathds{R}^2\setminus \overline{D}, \\
   &\lim\limits_{{r}\rightarrow\infty}\sqrt{r}(\partial u^s/\partial r-{i}ku^s)=0,
   \end{array}\right.
\end{equation}
where $r=|x|$. For the well-posedness of the above direct scattering problem, one needs to impose suitable conditions on $\partial D$,
which depend on the physical properties of the scatterer. The total field in \eqref{direct1}
satisfies
\begin{itemize}
\item[1)] the Dirichlet boundary condition
\begin{equation*}
u=0, \ \ \ \textrm{on}\  \partial D,
\end{equation*}
for a sound-soft obstacle;
\item[2)] the Neumann boundary condition
\begin{equation*}
\frac{\partial u}{\partial \nu}=0, \ \ \ \textrm{on}\  \partial D,
\end{equation*}
for a sound-hard obstacle,
where $\nu$ is the unit outward normal to $\partial D$; or
\item[3)] the impedance boundary condition
\begin{equation*}
\frac{\partial u}{\partial \nu}(x) +{i}\lambda u(x)=0,\ \ \ \textrm{on}\  \partial D,
\end{equation*}
for an impedance obstacle with some real-valued parameter $\lambda\geq 0$.
\end{itemize}

The scattering problem of time-harmonic acoustic waves for an inhomogeneous medium is modeled by
\begin{equation}\label{direct}
  \left\{
   \begin{array}{lll}
   &\Delta u+k^2n(x) u=0,\ \ \ \textrm{in}\  \mathds{R}^2 \\
   &u=u^i+u^s, \\
   &\lim\limits_{{r}\rightarrow\infty}\sqrt{r}(\partial u^s/\partial r-{i}ku^s)=0,
   \end{array}\right.
\end{equation}
where $n(x):=n_1(x)+i\frac{n_2(x)}{k}, n_1(x)>0, n_2(x)\geq 0$ is the refractive index. We assume that $n(x) \ne 1$ in $D$ and $n(x)=1$ in $\mathds{R}^2\setminus \overline{D}$.

It is known that the scattered fields $u^s$ of the above problems have an asymptotic expansion
\[
u^s(x)
=\frac{e^{i \frac{\pi}{4}}}{\sqrt{8k\pi}}\frac{e^{i kr}}{\sqrt{r}}\left\{u_{\infty}(\hat{x})+\mathcal{O}\left(\frac{1}{r}\right)\right\}\quad\mbox{as }\,r:=|x|\rightarrow\infty
\]
uniformly in all directions $\hat{x}=x/|x|$. The function $u_\infty(\hat{x})$ defined on $\mathds{S}:=\{\hat{x} | \hat{x} \in \mathds{R}^2, |\hat{x}|=1\}$ is known as the far-field pattern of $u^s$ due to
the incident field $u^i$.

The rest of this section is devoted to the far field pattern of a sound-soft disc, which serves as the kernel of the integral operator in the next section. Let $B \subset \mathds{R}^{2}$ be a disc centered at the origin with radius $R$. The far-field pattern corresponding to the incident plane wave with direction $d$ is given by
\cite{AbramowitzStegun1965}:
\begin{flalign}\label{uinf}
     u^B_{\infty}(\hat{x}; d)=-e^{-i \frac{\pi}{4}}\sqrt{\frac{2}{\pi k}}\bigg[\frac{J_0(kR)}{H_0^{(1)}(kR)}+2\sum_{n=1}^{\infty}\frac{J_n(kR)}{H_n^{(1)}(kR)}\cos(n\theta)\bigg], 
     \quad  \hat{x}\in \mathds{S},
  \end{flalign}
where $J_n$ is the Bessel function, $H_{n}^{(1)}$ is the Hankel function of the first kind of order $n$, $\theta =\angle (\hat{x},d)$, the angle
between $\hat{x}$ and $d$. Let 
\[
B_z:=\{x+z; x\in B, z \in \mathds{R}^2\}
\] 
be the shifted disc of $B$ and $u_{\infty,z}(\hat{x}; d)$ be the far field pattern of $B_z$. Then the translation property holds \cite{LiuZhang2017CMR, ZhangZhang2017JCP}:
\begin{equation}\label{shift}
u^{B_z}_{\infty}(\hat{x}; d)=e^{i kz\cdot (d-\hat{x})}u^B_{\infty}(\hat{x},d),\ \ \ \ \hat{x}\in\mathds{S}.
\end{equation}

\section{The Extended Sampling Method}\label{ESM}
The inverse scattering problem of interests is to reconstruct the location and approximate support of the scatterer without knowing the physical properties of the scatterer. We propose a novel method, called the Extended Sampling Method (ESM). It can process limited aperture far field pattern, full aperture far field pattern, or far field pattern of multiple frequencies.

In this paper, we mainly focus on the case when the far-field pattern is available of all observation directions but only one single incident wave, i.e., $u_\infty(\hat{x}; d_0)$ for all $\hat{x} \in \mathds{S}$ and a fixed $d_0 \in \mathds{S}$. The unique determination of the scatterer by the far field pattern of one single incident wave is a long-standing open problem in the inverse scattering theory. The answer is only partially known for some special  scatterers. Some numerical methods have been proposed in the past few years, e.g., the range test method \cite{PotthastSylvestrKusiak2003IP}, the no-response method \cite{Potthast2006IP}, the orthogonality sampling method \cite{Potthast2010IP}, the direct sampling method \cite{ItoJinZou2012IP}, and the iterative decomposition method \cite{LiuDai2015JCAM}.

\subsection{The Far Field Equation}
Let $u^{B_z}_{\infty}(\hat{x}; d), x\in \mathds{S}$ be the far-field pattern of the sound-soft disc $B_z$ centered at $z$ with radius $R$
for plane incident waves of all directions $d\in \mathds{S}$.
Define a far field operator $\mathcal{F}_z: L^2(\mathds{S}) \to L^2(\mathds{S})$
\begin{equation}\label{FO}
\mathcal{F}_z g(\hat{x}) = \int_{\mathds{S}}u^{B_z}_{\infty}(\hat{x},d)g(d)d s(d), \quad \hat{x} \in \mathds{S}.
\end{equation}

Let $U^s(x)$ and $U_\infty(\hat{x})$ be the scattered field and far field pattern of the scatterer $D$ due to one incident wave, respectively.
Using the far field operator $\mathcal{F}_z$, we set up a far field equation
\begin{equation}\label{fe}
\left(\mathcal{F}_z g\right)(\hat{x})=U_\infty(\hat{x}),\ \ \ \ \hat{x}\in \mathds{S}.
\end{equation}
This integral equation is the main ingredient of the ESM. 
The advantage of using $\mathcal{F}_z$ is that it can be computed easily while the classical far field operator uses full-aperture measured far field pattern.

We expect that the (approximate) solution of \eqref{fe} for a sampling point $z$
would provide useful information for the reconstruction of $D$.
To this end, we first introduce some results that are useful in analyzing the far field equation \eqref{fe} (see Corollary 5.31, Corollary 5.32 and Theorem 3.22 of \cite{ColtonKress2013}).
\begin{lemma}\label{lemma1}
The Herglotz operator $\mathcal{H}: L^2(\mathds{S})\rightarrow H^{{1}/{2}}(\partial B_z)$ defined by
\begin{equation}
(\mathcal{H}g)(x):=\int_{\mathds{S}} e^{ikx\cdot d}g(d)ds(d),\ \ \ \ x\in \partial B_z
\end{equation}
is injective and has a dense range provided $k^2$ is not a Dirichlet eigenvalue for the negative Laplacian for $B_z$.
\end{lemma}

\begin{lemma}\label{lemma2}
The operator $A:H^{1/2}(\partial B_z)\rightarrow L^2(\mathds{S})$ which maps the boundary values of radiating solutions $u\in H_{loc}^1(\mathds{R}^2\setminus\overline{B}_z)$ of the Helmholtz equation onto the far field pattern $u_\infty$ is bounded, injective and has a dense range.
\end{lemma}

\begin{lemma}\label{lemma3}
The far field operator $\mathcal{F}_z$ is injective and has dense range provided $k^2$ is not a Dirichlet eigenvalue for the negative Laplacian for $B_z$.
\end{lemma}

The following theorem is the main result for the far field equation \eqref{fe}.
\begin{theorem}\label{theorem1}
Let $B_z$ be a sound soft disc centered at $z$ with radius $R$. Let $D$ be an inhomogeneous medium or an obstacle with the Dirichlet, Neumann, or impedance boundary condition. Assume that $kR$ does not coincide with any zero of the Bessel functions $J_n, n=0,1,2,\cdots$. Then the following results hold for the far field equation (\ref{fe}):
\begin{itemize}
\item[1.] If $D\subset B_z$, for a given $\varepsilon>0$, there exists a function $g_z^\varepsilon\in L^2(\mathds{S})$ such that
\begin{equation}\label{fe2}
\bigg\|\int_{\mathds{S}}u^{B_z}_{\infty}(\hat{x},d)g_z^\varepsilon(d)d s(d)-U_\infty(\hat{x})\bigg\|_{L^2(\mathds{S})}<\varepsilon
\end{equation}
and the Herglotz wave function $v_{g_z^\varepsilon}(x):=\int_{\partial B_z}e^{ikx\cdot d}g_z^\varepsilon (d)ds(d), x\in B_z$ converges to the solution $w\in H^1(B_z)$ of the Helmholtz equation with $w=-U^s$ on $\partial B_z$ as $\varepsilon\rightarrow 0$.

\item[2.] If $D\cap B_z=\emptyset$, every $g_z^\varepsilon\in L^2(\mathds{S})$ that satisfies (\ref{fe2}) for a given $\varepsilon>0$ is such that
\begin{equation}
\lim_{\varepsilon\rightarrow 0}\|v_{g_z^\varepsilon}\|_{H^1(B_z)}=\infty.
\end{equation}

\item[3.] If $D\cap B_z\neq\emptyset$ and $D \not\subset B_z$. When the scattered solution of $D$ can be extended from $\mathds{R}^2\setminus D$ into $\mathds{R}^2\setminus (D\cap B_z)$, then the conclusion is the same with Case 1; Otherwise, the conclusion is the same with Case 2.

\end{itemize}
\end{theorem}

\begin{proof}
Since $kR$ does not coincide with any zero of the Bessel functions $J_n, n=0,1,2,\cdots$, 
then $k^2$ is not a Dirichlet eigenvalue for the negative Laplacian for $B_z$:
\begin{equation}\label{dep}
 \left\{
  \begin{array}{ll}
  \Delta u+k^2 u=0,\ \ \ \textrm{in}\  B_z, \\
  u=0,\ \ \ \textrm{on} \ \partial B_z.
  \end{array}\right.
\end{equation}

Case 1: Let $D\subset B_z$. From \autoref{lemma1}, for any $\varepsilon$, we have $g_z^\varepsilon$ such that
\begin{equation}\label{equ2.1}
\|\mathcal{H}g_z^\varepsilon+U^s\|_{L^2(\partial B_z)}\leq \frac{\varepsilon}{\|A\|},
\end{equation}
where $A:H^{1/2}(\partial B_z)\rightarrow L^2(\mathds{S})$ is defined in \autoref{lemma2}. Then
\begin{equation}\label{equ2.2}
A(-\mathcal{H}g_z^\varepsilon)=\int_\mathds{S}u^{B_z}_{\infty}(\cdot,d)g_z^\varepsilon(d)d s(d).
\end{equation}
Since $D\subset B_z$, we also have
\begin{equation}\label{equ2.3}
A(U^s)=U_\infty.
\end{equation}
Taking the difference of (\ref{equ2.2}) and (\ref{equ2.3}), and from \autoref{lemma2} and (\ref{equ2.1}), for any $\varepsilon$, we have $g_z^\varepsilon$ such that
\begin{eqnarray}
&&\bigg\|\int_{\mathds{S}}u^{B_z}_{\infty}(\hat{x},d)g_z^\varepsilon(d)d s(d)-U_\infty(\hat{x})\bigg\|_{L_2(\mathds{S})}\nonumber\\
&=&\|A(-\mathcal{H}g_z^\varepsilon)-A(U^s)\|_{L^2(\mathds{S})}\nonumber\\
&\leq& \|A\|\cdot \|\mathcal{H}g_z^\varepsilon+U^s\|_{L^2(\mathds{S})}\nonumber\\
&\leq& \varepsilon.
\end{eqnarray}
By (\ref{equ2.1}), as $\varepsilon\rightarrow 0$, the Herglotz wave function 
$v_{g_z^\varepsilon}(x):=\int_{\partial B_z}e^{ikx\cdot d}g_z^\varepsilon (d)ds(d), x\in B_z$ 
converges to the unique solution of the following well-posed inner Dirichlet problem $w\in H^1(B_z)$
\begin{equation*}
 \left\{
  \begin{array}{ll}
  \Delta w+k^2 w=0,\ \ \ \textrm{in}\  B_z, \\
  w=-U^s,\ \ \ \textrm{on} \ \partial B_z.
  \end{array}\right.
\end{equation*}

Case 2: Let $D\cap B_z=\emptyset$. From  \autoref{lemma3}, for every $\varepsilon$, there exists $g_z^\varepsilon\in L^2(\mathds{S})$ such that
\begin{equation}\label{equ2.4}
\bigg\|\int_\mathds{S} u^{B_z}_{\infty}(\hat{x},d)g_z^\varepsilon(d)ds(d)-U_\infty(\hat{x})\bigg\|_{L_2(\mathds{S})}<\varepsilon.
\end{equation}
Assume to the contrary that there exists a sequence $v_{g_z^{\varepsilon_n}}$ 
such that $\|v_{g_z^{\varepsilon_n}}\|_{H^1(B_z)}$ remains bounded as $\varepsilon_n \to 0, n \to \infty$. 
Without loss of generality, we assume that $v_{g_z^{\varepsilon_n}}$ converges to $v_{g_z}\in H^1(B_z)$ weakly as $n\rightarrow \infty$, where $v_{g_z}(x)=\int_{\partial B_z}e^{ikx\cdot d}g_z (d)ds(d), x\in B_z$. Let $V^s\in H_{loc}^1(\mathds{R}^2\setminus \overline{B}_z)$ be the unique solution of the following exterior Dirichlet problem
\begin{equation*}
  \left\{
   \begin{array}{lll}
   &\Delta V^s+k^2V^s=0\ \ \ \textrm{in}\  \mathds{R}^2\setminus \overline{B}_z, \\
   &V^s=-v_{g_z} \ \ \ \textrm{on}\  \partial B_z,\\
   &\lim\limits_{r\rightarrow\infty}\sqrt{r}(\partial V^s/\partial r-{i}kV^s)=0.
   \end{array}\right.
\end{equation*}
Its far field pattern is $V_\infty=\int_\mathds{S} u^{B_z}_{\infty}(x,d)g_z(d)ds(d)$. While from (\ref{equ2.4}), as $\varepsilon\rightarrow 0$, we have $\int_\mathds{S} u^{B_z}_{\infty}(x,d)g_z(d)ds(d)=U_\infty$. Consequently,
\begin{equation}
V_\infty=U_\infty.
\end{equation}
Then by Rellich's lemma (see Lemma 2.12 in \cite{ColtonKress2013}), the scattered fields coincide in $\mathds{R}^2\setminus (\overline{B}_{z}\cup \overline{D})$ and we can identify $\tilde{U}^s:=V^s=U^s$ in $\mathds{R}^2\setminus (\overline{B}_{z}\cup \overline{D})$. We have that $V^s$ has an extension into $\mathds{R}^2\setminus B_z$ and $U^s$ has an extension into $\mathds{R}^2\setminus D$. Since $D\cap B_z=\emptyset$, $\tilde{U}^s$ can be extended from $\mathds{R}^2\setminus (\overline{B}_{z}\cup \overline{D})$ into all of $\mathds{R}^2$, that is, $\tilde{U}^s$ is an entire solution to the Helmholtz equation. Since $\tilde{U}^s$ also satisfies the radiation condition, it must vanish identically in all of $\mathds{R}^2$. This leads to a contradiction since $\tilde{U}^s$ is not a null function.

Case 3: If the scattered field of $D$ can be extended from $\mathds{R}^2\setminus D$ into $\mathds{R}^2\setminus (D\cap B_z)$, then $\tilde{D}:=D\cap B_z$ has the far field pattern $U_\infty$ and $\tilde{D}\in B_z$. One can simply replace $D$ with $\tilde{D}$ in Case 1 to arrive the same conclusion. Otherwise, if the scattered solution of $D$ can not be extended from $\mathds{R}^2\setminus D$ into $\mathds{R}^2\setminus (D\cap B_z)$, a similar proof to Case 2 works.
\end{proof}
\begin{remark}
Note that the radius $R$ of the disc $B_z$ can be chosen such that $kR$ is not a Dirichlet eigenvalue for the negative Laplacian for $B_z$. Hence the 
ESM is independent of wavenumbers, i.e., the ESM works for all wavenumbers. In contrast, the LSM needs to exclude wavenumbers which are certain eigenvalues related to the scatterer.
\end{remark}

\subsection{The ESM Algorithm}
Now we are ready to present the extended sampling method (ESM) to reconstruct the location and approximate support of a scatterer $D$ using the far field pattern due to a single incident wave.

%

Recall that $u^{B_z}_{\infty}(\cdot,d)$ is the far field pattern of a sound soft disc $B_z$ centered at $z$ with radius $R$ and $U_\infty(\hat{x})$ is the measured far field pattern of an unknown scatterer $D$ due to an incident plane wave. 
Let $\Omega$ be a domain containing $D$. For a sampling point $z \in \Omega$, we consider the linear ill-posed integral equation 
\begin{equation}\label{ffe}
\int_{\mathds{S}}u^{B_z}_{\infty}(\hat{x},d)g_z(d)d s(d)=U_\infty(\hat{x}),\ \ \ \ \hat{x}\in \mathds{S}.
\end{equation}
Suppose that \eqref{ffe} is solved by some regularization scheme, say Tikhonov regularization with parameter $\alpha$.
By \autoref{theorem1}, one expects that the approximate solution $\|g^\alpha_z\|_{L^2(\mathds{S})}$ is relatively large when $D$ is not inside $B_z$
and relatively small when $D$ is inside $B_z$.
Consequently, an approximation of the location and support of the scatterer $D$ can be obtained by plotting $\|g^\alpha_z\|_{L^2(\mathds{S})}$ for
all sampling points $z \in \Omega$.


\begin{figure}[h!]
\centering
\includegraphics[width=0.7\textwidth]{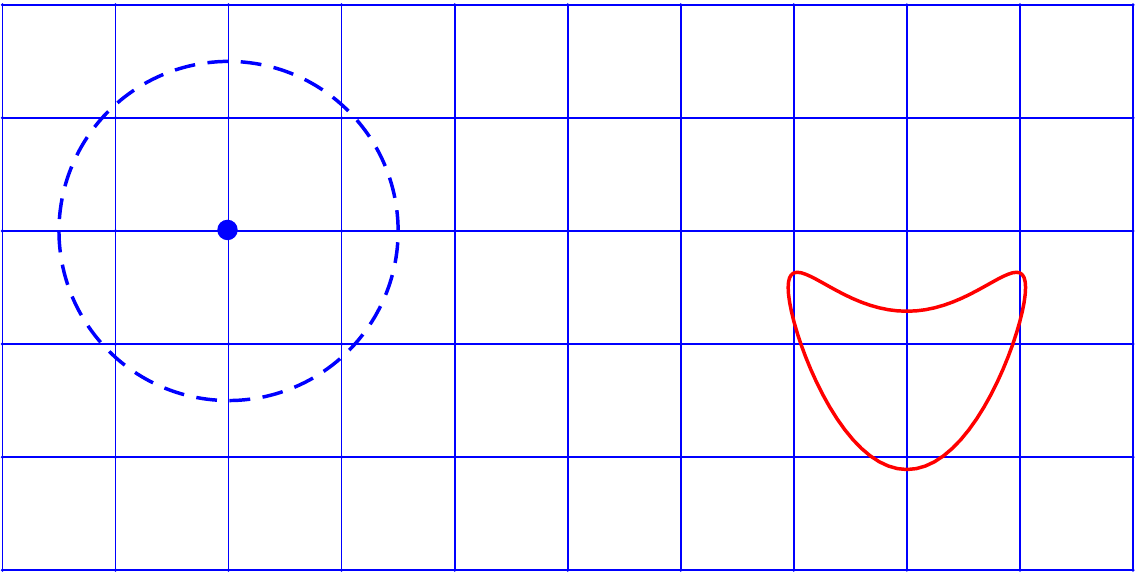}
\caption{\label{fig1}Illustration of the extended sampling method.}
\end{figure}

The algorithm of the ESM is as follows (see \autoref{fig1}).
\begin{itemize}
\item[ ] \hspace{-1cm}{\bf{The Extended Sampling Method}}
\item[1.] Generate a set $T$ of sampling points for $\Omega$ which contains $D$.
\item[2.] Compute $u^{B_z}_{\infty}(\hat{x},d)$ for all $\hat{x}\in \mathds{S}$ and $d \in \mathds{S}$
\item[3.] For each sampling point $z \in T$, 
	\begin{itemize}
	\item[a.] set up a discrete version of \eqref{ffe};
	\item[b.] use the Tikhonov regularization to compute an approximate $g^\alpha_z$ to \eqref{ffe} ($\alpha$ is chosen according to the Morozov discrepancy principle);
	\end{itemize}
\item[4.] Find the global minimum point $z^*\in T$ for $\|g^\alpha_{z}\|_{L^2}$.
\item[5.] Choose $B_{z^*}$ to be the approximate support for $D$.
\end{itemize}

\begin{remark}
Let $\tilde{D}$ be the smallest subset of $D$ such that the scattered field for $D$ can be extended from $\mathds{R}^2\setminus D$ into $\mathds{R}^2\setminus \tilde{D}$. If the radius $R$ of the sampling disc is greater than the radius of the circumscribe circle of $\tilde{D}$, then $B_{z^*}$ found by the ESM method should contain $\tilde{D}$.
\end{remark}


\subsection{Multilevel ESM}
An important step for the ESM is to choose the radius $R$ of $B_z$.
It would be ideal to choose $B_z$ to be slightly larger than the scatterer $D$.
However, this is not possible if no a priori information about the scatterer is available.
To resolve the difficulty, we propose a multilevel technique to choose a suitable radius $R$.
\begin{itemize}
\item[ ]\hspace{-1cm}{\bf{Multilevel ESM}}
\item[1.] Choose the sampling discs with a large radius $R$. Generate $T$ such that the distance between sampling points is roughly $R$. 
Using the ESM, determine the global minimum point $z_0\in T$ for  $\|g^\alpha_{z}\|_{L^2}$ and an approximation $D_0$ for $D$.

\item[2.] For $j=1,2,\cdots$
	\begin{itemize}
	\item Let $R_j=R/2^j$ and generate $T_j$ with the distance between sampling points being roughly $R_j$. 
	\item Find the minimum point $z_j\in T$.  If $z_{j} \not\in D_{j-1}$, go to Step 3.
	\end{itemize}

\item[3.] Choose $z_{j-1}$ and $D_{j-1}$ to be the location and approximate support of $D$, respectively.
\end{itemize}

The second step of the Multilevel ESM  is a loop. 
The following is some heuristic argument on the termination of the loop.
If the size of the sampling disc $R_j$ is small enough, the scattered field cannot be continuously extended to the complement of
any of the sampling discs according to \autoref{theorem1}. Then the minimum point $z_j$ should lie outside $D$
and the smallest proper size of the sampling discs is found. We shall see that numerical examples support this argument.

\subsection{Relation to Other Methods}
The far field equation \eqref{fe} of the ESM looks similar to the linear sampling method (LSM) proposed by Colton and Kirsch \cite{ColtonKirsch1996IP}.
The far field equation for the LSM is 
\begin{equation}\label{FELSM}
\int_{\mathds{S}}U_\infty(\hat{x}; d)g_z(d)ds(d)=\Phi_\infty(\hat{x}; z), \ \ \ \ \hat{x}\in \mathds{S}.
\end{equation}
In \eqref{FELSM}, $U_\infty(\hat{x};d), \hat{x}, d\in \mathds{S}$ is the full-aperture far field pattern of the scatterer $D$ due to the incident plane wave $u^i=\exp(ikx\cdot d)$, and $\Phi_\infty(\hat{x}, z)$ is the far field pattern of the fundamental solution of the Helmholtz equation $\Phi(x,z):=\frac{i}{4}H^{(1)}_0(k|x-z|)$,
where $H^{(1)}_0$ is the Hankel function of the first kind of order zero. For each sampling point $z$ in the sampling region $\Omega$ which contains $D$,
using some appropriate regularization scheme, one obtains an approximate solution $g^\alpha_z$.
In general, the norm $\|g^\alpha_z\|$ is larger for $z\not \in D$ and smaller for $z\in D$ such that the support of $D$ can be reconstructed accordingly.

Since $U_\infty(\hat{x}; d)$ is used as the kernel of the integral operator in \eqref{FELSM}, full aperture far field pattern is necessary. Namely, one needs
 $U_\infty(\hat{x}; d)$ for all $\hat{x}\in \mathds{S}$ and $d \in \mathds{S}$ to set up \eqref{FELSM}.
 As a consequence, the LSM cannot directly process limited aperture far field pattern, e.g., $U_\infty(\hat{x}; d_0)$ for a single incident direction $d_0$.
 This also applies to other classical sampling methods such as the Factorization Method and the Reciprocity Gap Method.

 In contrast, the kernel of the integral operator $\mathcal{F}_z$ for the ESM is the full aperture far field pattern $u_{\infty, z}(\hat{x};d)$ of a sound soft disc $B_z$. The measured far field pattern $U_{\infty}(\hat{x})$ is the right hand side of the far field equation \eqref{fe}. Hence,
 in principle, the ESM can take limited aperture data of any type as input.

Another related method is the range test due to Potthast, Sylvester and Kusiak \cite{PotthastSylvestrKusiak2003IP}. 
Similar to the ESM, it processes $U_\infty(\hat{x})$ for a single incident wave.
Let $G$ be a convex test domain. The integral equation in \cite{PotthastSylvestrKusiak2003IP} is given by
\begin{equation}\label{RTIE}
\int_{\partial G} e^{i k \hat{x} \cdot{d}} g(d) ds(d) = U_\infty(\hat{x}), \quad \hat{x} \in \mathds{S}.
\end{equation}
It has a solution in $L^2(\partial G)$ if and only if the scattered field can be analytically extended up to the boundary $\partial G$.
Then a convex support can be obtained numerically and the intersection of these supports provides information to reconstruct $D$.

The ESM uses a slightly different integral equation. More importantly, other than finding a convex domain, the ESM computes an indicator for a sampling point,
which makes it possible to process data of multiple directions or even of multiple frequencies easily.

\section{Numerical Examples}\label{NEs}
We now present some numerical examples to show the performance of the ESM.
In particular, we consider the inverse scattering problems for inhomogeneous media and impenetrable obstacles with Dirichlet, Neumann and impedance boundary conditions. 
The synthetic far-field data is generated using boundary integral equations \cite{ColtonKress2013} for impenetrable obstacles.
For inhomogeneous media, a finite element method with perfectly matched layer (PML) \cite{ChenLiu2005SIAM} is used to
compute the scattered data. Then the data are used to obtain the far-field pattern \cite{Monk1995COMPEL}.

We consider two obstacles: a triangle whose boundary is given by
\begin{equation*}
(1+0.15\cos 3t)\big(\cos t, \sin t\big)+\big(3,5\big),\ \ \ t\in [0,2\pi)
\end{equation*}
and a kite whose boundary is given by
\begin{equation*}
\big(1.5\sin t, \cos t+0.65\cos 2t-0.65\big)+\big(3,5\big),\ \ \ t\in [0,2\pi).
\end{equation*}
In the case of the impedance boundary condition, we set $\lambda=2$. 

For the inhomogeneous medium, $D$ is the L-shaped domain given by
\begin{equation*}
\big[-0.9,1.1\big]\times \big[-1.1,0.9\big]\setminus \big[0.1,1.1\big]\times \big[-1.1,-0.1\big]
\end{equation*}
with the refractive index
 \begin{displaymath}
  n(x,y)= \left\{
     \begin{array}{lr}
       \frac{1}{2}\big(3+\cos(2\pi\sqrt{x^2+y^2})\big), & (x, y)^T\in D,\\
       1, & \textrm{otherwise}.
     \end{array}
   \right.
\end{displaymath}

Using (\ref{shift}), we can rewrite the far-field equation (\ref{ffe}) as 
\[
\int_{\mathds{S}}e^{i kz\cdot (d-\hat{x})}u^B_{\infty}(\hat{x},d)g_z(d)d s(d) = U_\infty(\hat{x}),\ \ \ \ \hat{x}\in \mathds{S},
\]
i.e.,
\begin{equation}\label{ffe2}
\int_{\mathds{S}}e^{i kz\cdot d}u^B_{\infty}(\hat{x},d)g_z(d)d s(d) = e^{i kz\cdot \hat{x}}U_\infty(\hat{x}),\ \ \ \ \hat{x}\in \mathds{S},
\end{equation}
where $u^B_{\infty}(\hat{x};d)$ is given in \eqref{uinf}. 

The incident plane wave is $u^i(x)=e^{i kx\cdot d}$ with $d=(1,0)^T$ and $k=1$. For all the examples, 
the synthetic data is a $52\times 1$ vector $F^z$ such that 
$F^z_{j} \approx e^{i kz\cdot \hat{x}_j} U_\infty(\hat{x}_j, d)$, where  $U_\infty(\hat{x}_j, d)$ is far field pattern of 
$52$ observation directions $\hat{x}_j, j=1, 2, \cdots, 52,$ uniformly distributed on the unit circle. Then $3\%$ of uniform random noise is added to $U_\infty(\hat{x}_j, d)$.

\subsection{Examples for the ESM}
Let $\Omega=[-10,10]\times [-10, 10]$ and choose the samplings points to be 
\[
T:=\{(-10+0.1m, -10+0.1n), \quad m,n=0, 1, \cdots, 200\}.
\] 
The radius of the sampling discs is set to $R=1$. 

For each mesh point $z \in T$, we use the Tikhonov regularization with a fixed parameter $\alpha=10^{-5}$. 
We discretize (\ref{ffe2}) to obtain a linear system $A^zg_z=F^z$, where $A^z$ is the matrix given by
\[
A^z_{l,j}=e^{i kz\cdot d_j}u^B_{\infty}(d_l,d_j), \quad l, j=1,2,\cdots,52.
\] 
Then the regularized solution is given by
\begin{equation*}
g^\alpha_z\approx \big((A^z)^*A^z+\alpha A^z\big)^{-1}(A^z)^*F^z.
\end{equation*}
We plot the contours for the indicator function
\begin{equation}\label{Iz}
 I_z = \frac{\|g^\alpha_z\|_{l^2}}{\max\limits_{z\in T}\|g^\alpha_z\|_{l^2}}
\end{equation}
for all the sampling points $z\in T$.

\autoref{fig3}(a)  is the triangular scatterer. \autoref{fig3}(b), (c), (d) are the contour plots of $I_z$ for obstacles with
Dirichlet,  Neumann, and impedance boundary conditions, respectively.
The asterisks are the minimum points of $I_z$. 
\begin{figure}[h!]
\begin{minipage}{0.40\linewidth}
\includegraphics[angle=0, width=0.86\textwidth]{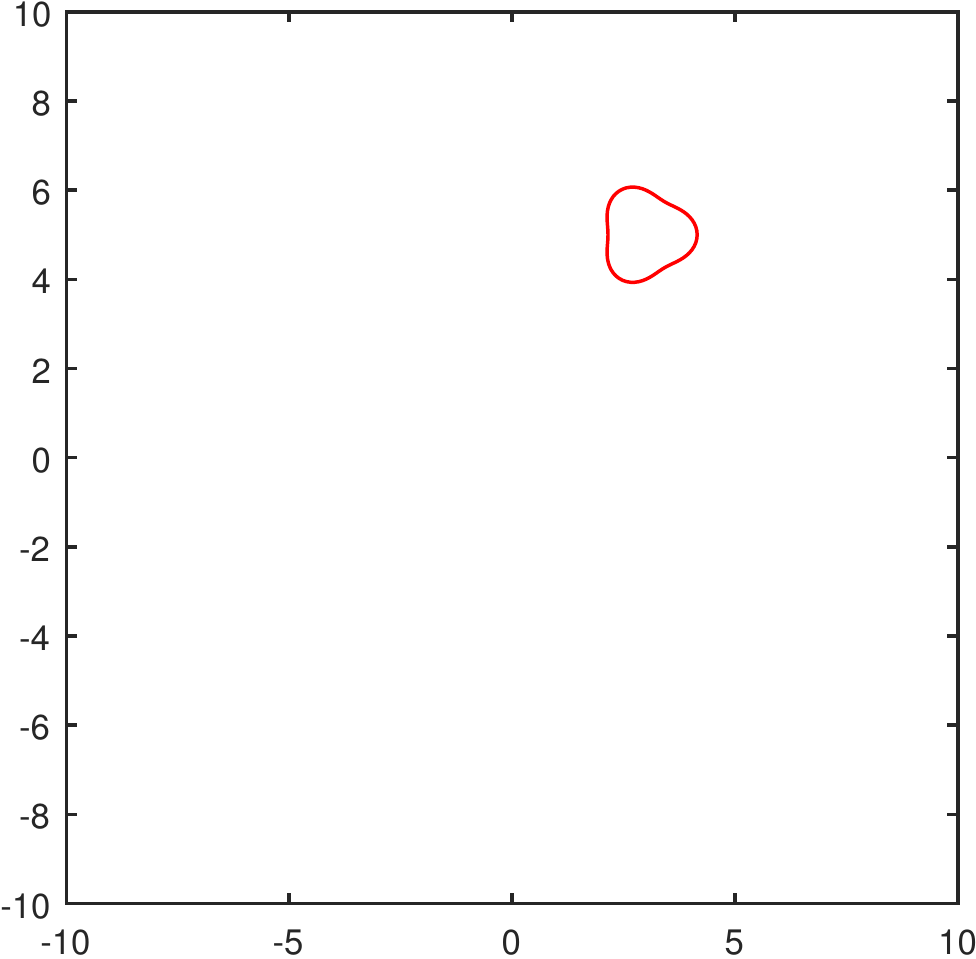}
\vspace{-0.5cm}
\begin{center}
(a)
\end{center}
\end{minipage}
\begin{minipage}{0.40\linewidth}
\includegraphics[angle=0, width=\textwidth]{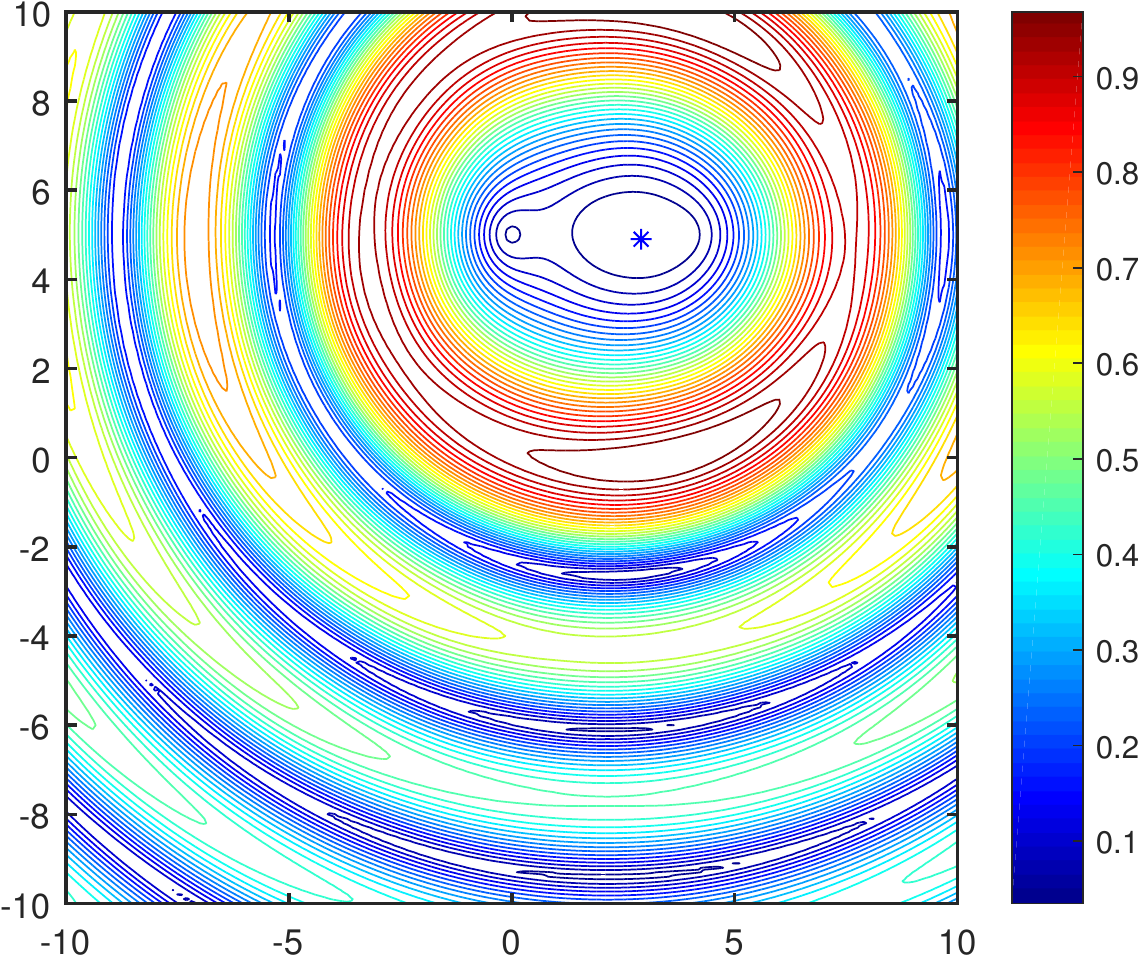}
\vspace{-1cm}
\begin{center}
(b)
\end{center}
\end{minipage}

\begin{minipage}{0.40\linewidth}
\includegraphics[angle=0, width=\textwidth]{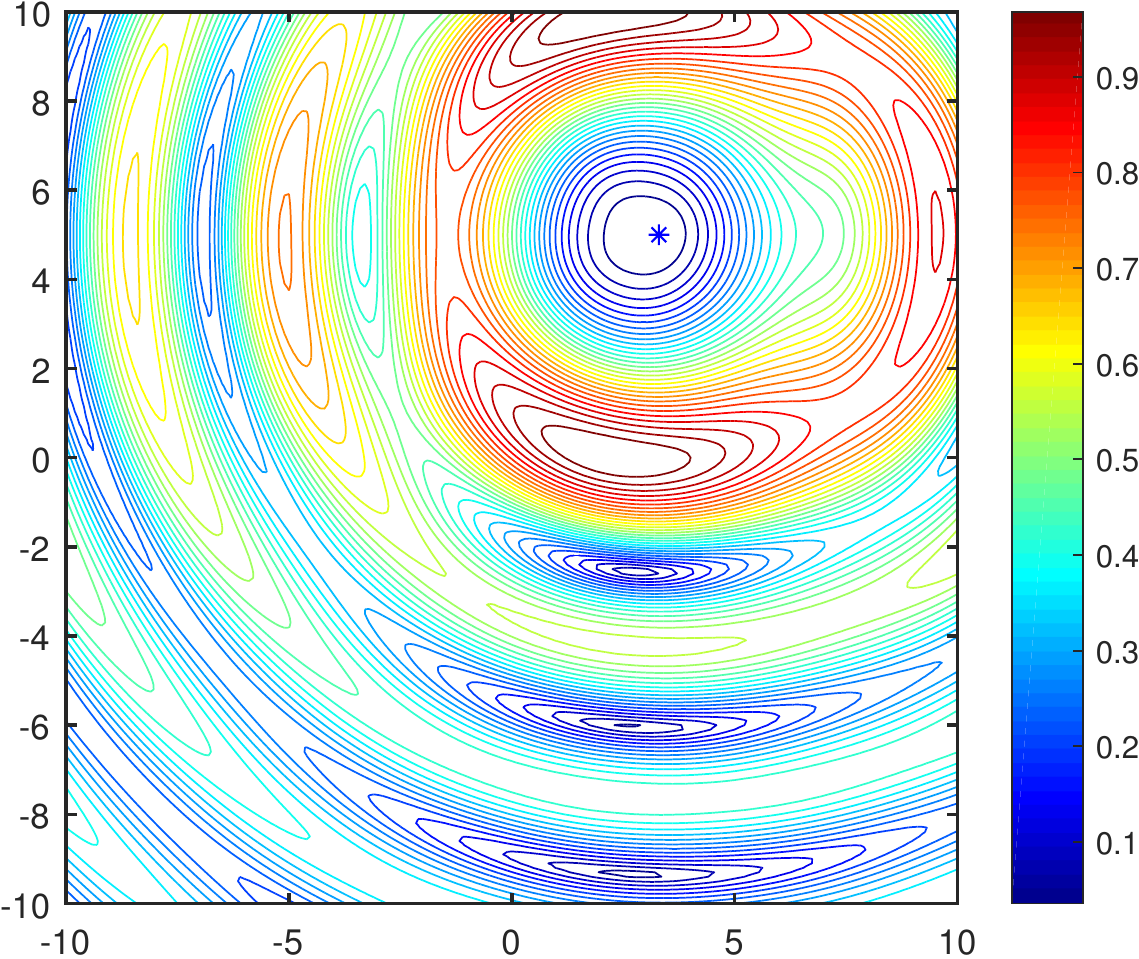}
\vspace{-1cm}
\begin{center}
(c)
\end{center}
\end{minipage}
\begin{minipage}{0.40\linewidth}
\includegraphics[angle=0, width=\textwidth]{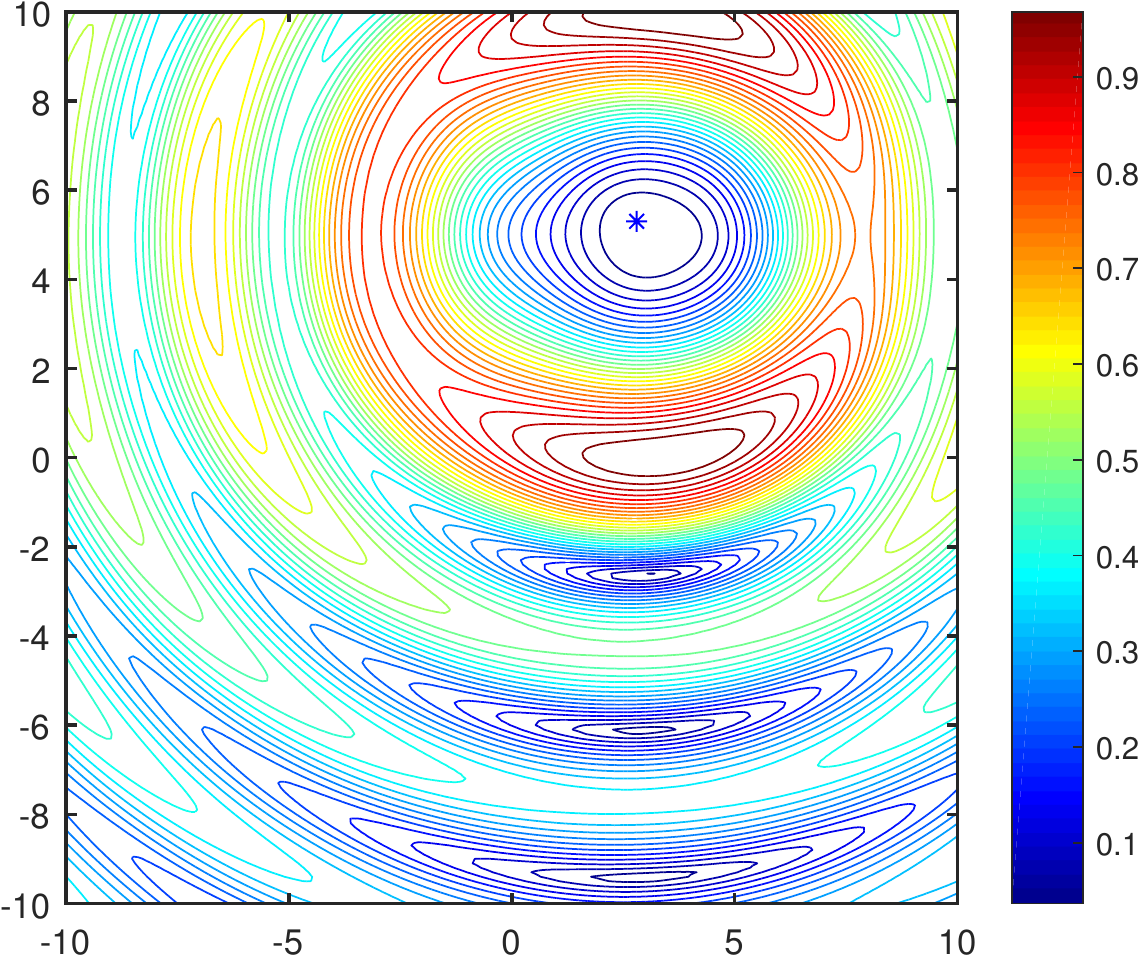}
\vspace{-1cm}
\begin{center}
(d)
\end{center}
\end{minipage}
\caption{\label{fig3} The triangle (a) and contour plots of $I_z$ using sampling discs of radius $R=1$: (b) Dirichlet BC, (c) Neumann BC, (d) Impedance BC.}
\end{figure}

We plot discs whose centers are $z$'s minimizing $I_z$ in \autoref{fig2} for different boundary conditions.
These minimization points provide the correct locations of the obstacles. 
\begin{figure}[h!]
\begin{minipage}{0.30\linewidth}
\includegraphics[angle=0, width=\textwidth]{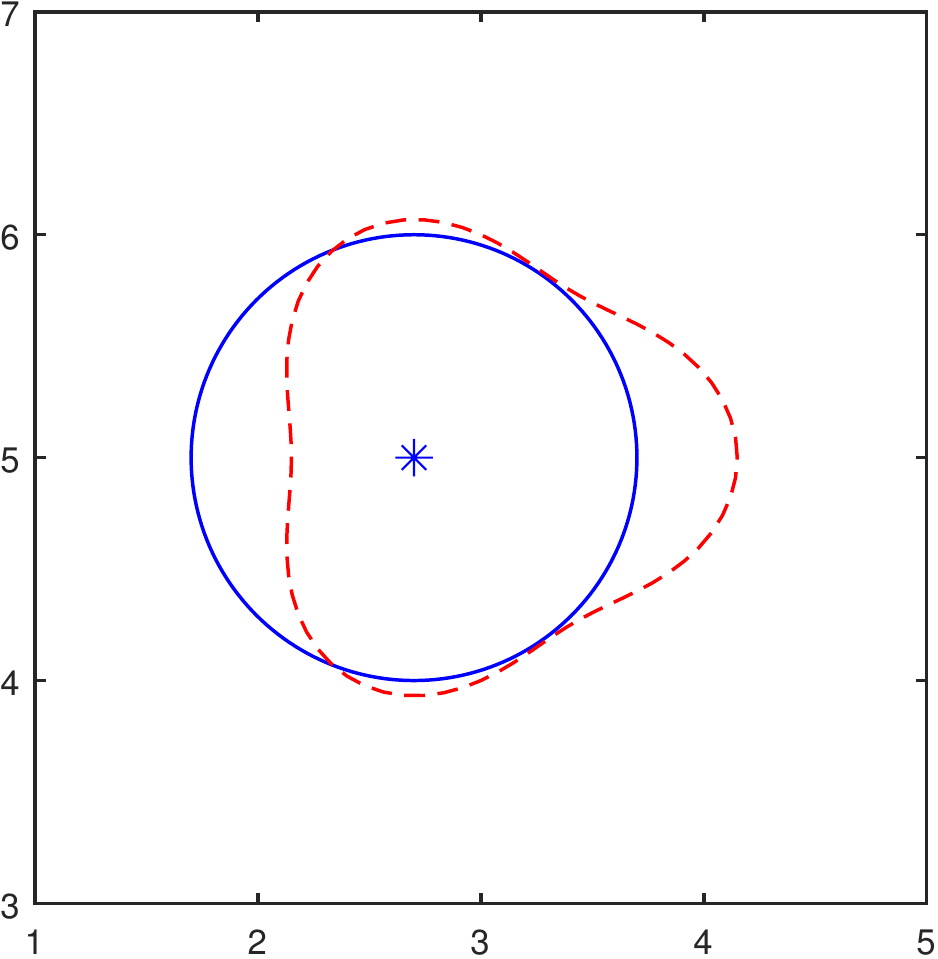}
\end{minipage}
\hspace{0.2cm}
\begin{minipage}{0.30\linewidth}
\includegraphics[angle=0, width=\textwidth]{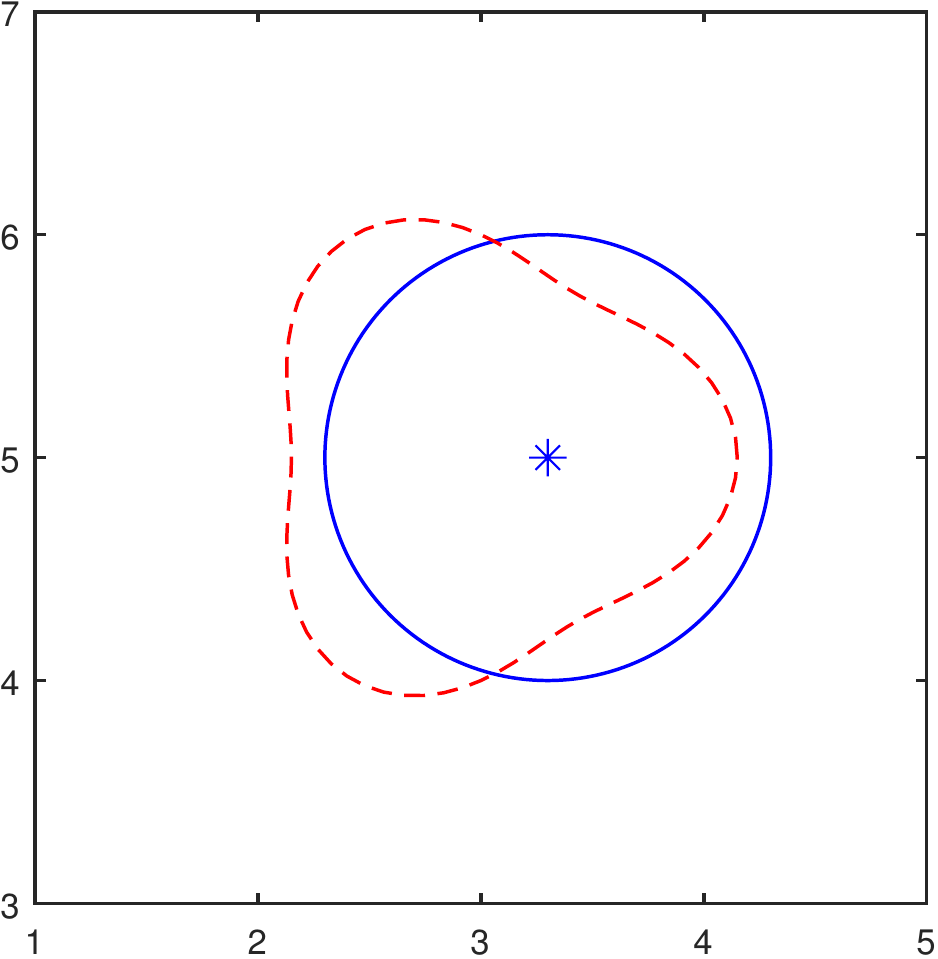}
\end{minipage}
\hspace{0.2cm}
\begin{minipage}{0.30\linewidth}
\includegraphics[angle=0, width=\textwidth]{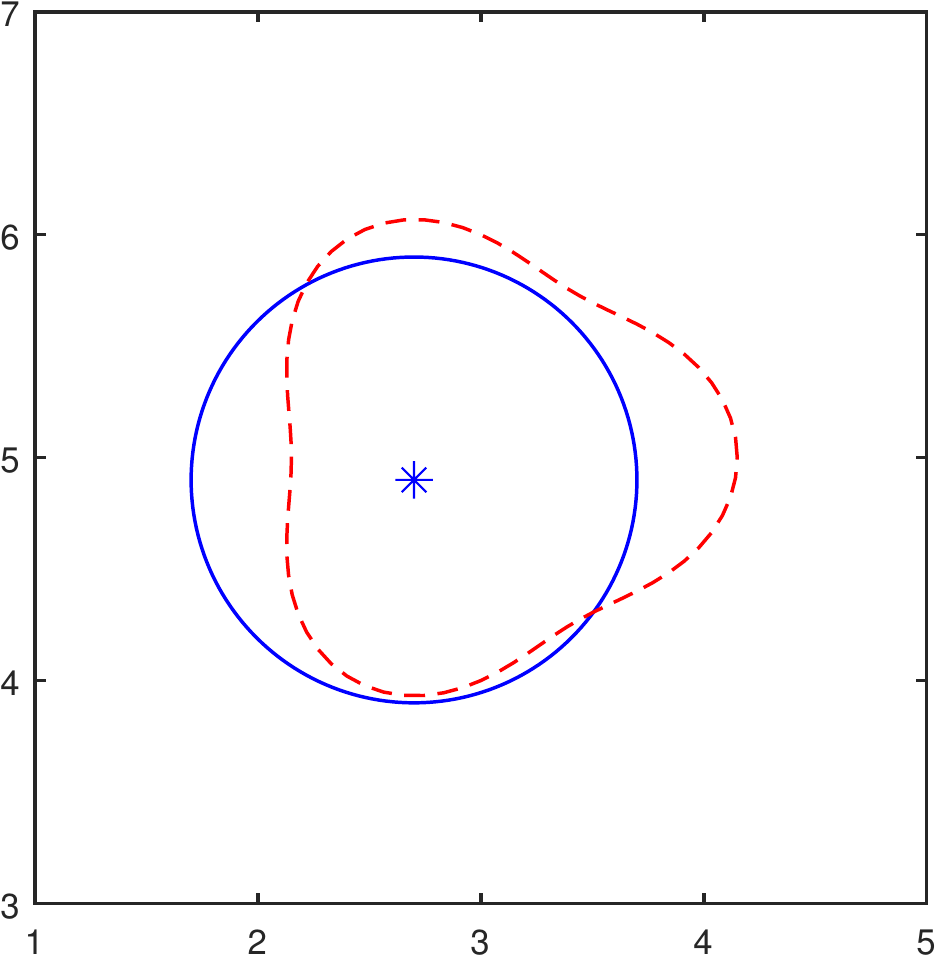}
\end{minipage}
\caption{\label{fig2} Approximate reconstructions for the triangle using sampling discs $R=1$. (a) Dirichlet boundary condition. (b) Neumann boundary condition. (c) Impedance boundary condition.}
\end{figure}


\autoref{fig4} show the similar results of the EMS for the kite with Dirichlet, Neumann and the impedance boundary conditions. 
\autoref{fig6} shows the reconstructions of the kite of different boundary conditions.  \autoref{fig7} show the results for the L-shape medium.
\begin{figure}[h!]
\begin{minipage}[t]{0.40\linewidth}
\includegraphics[angle=0, width=0.86\textwidth]{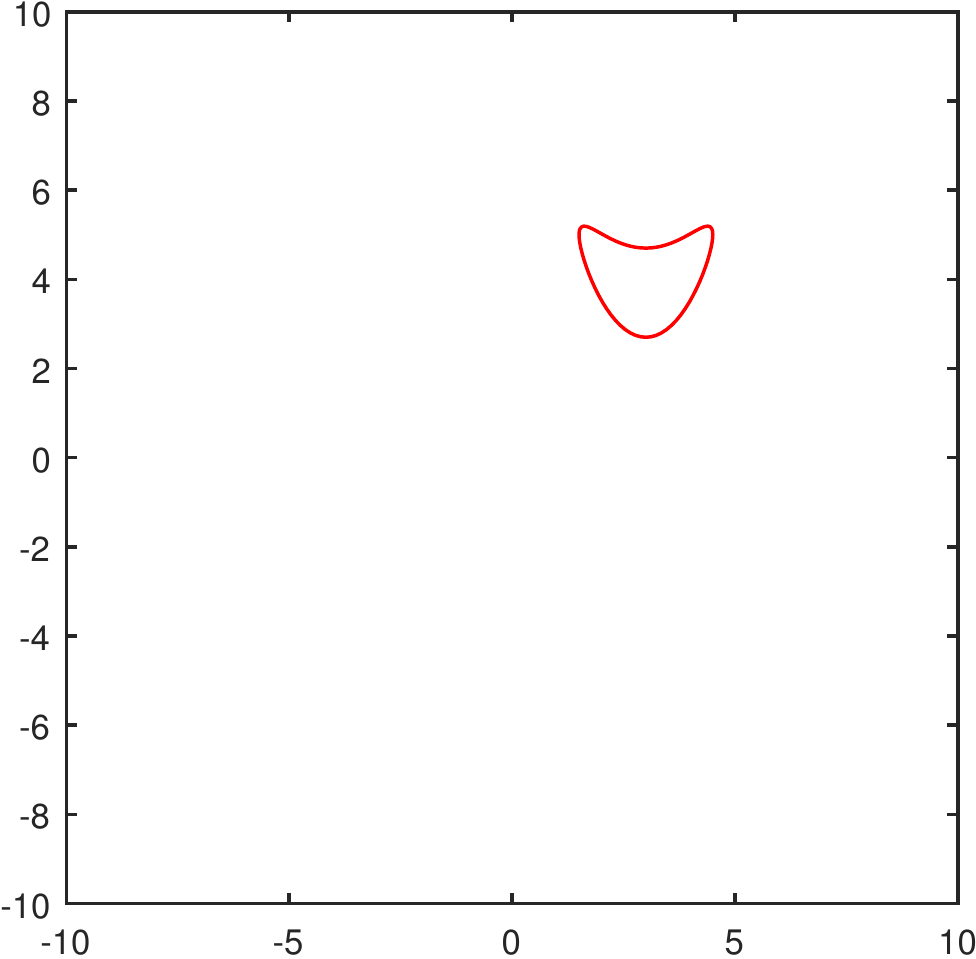}
\vspace{-0.5cm}
\begin{center}
(a)
\end{center}
\end{minipage}
\begin{minipage}[t]{0.40\linewidth}
\includegraphics[angle=0, width=\textwidth]{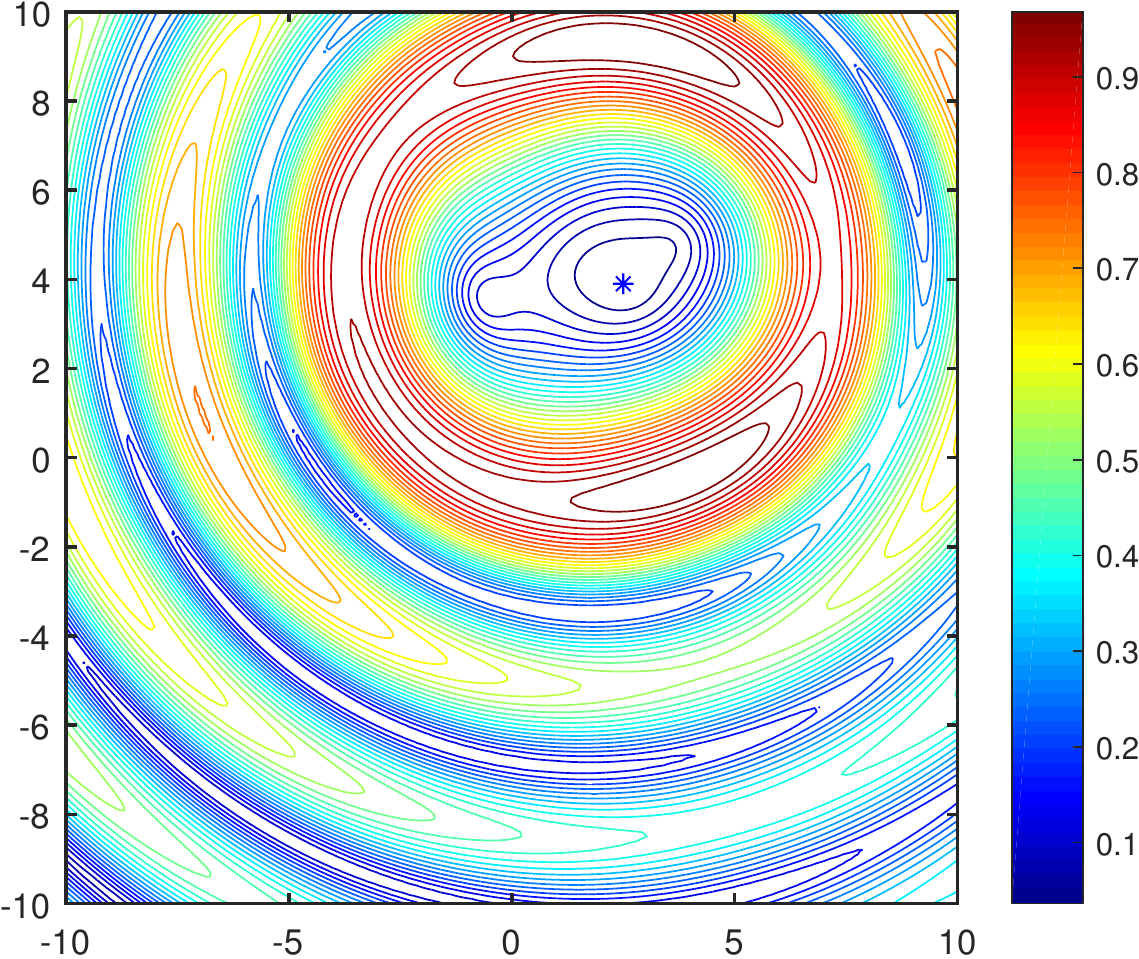}
\vspace{-1cm}
\begin{center}
(b)
\end{center}
\end{minipage}

\begin{minipage}[t]{0.40\linewidth}
\includegraphics[angle=0, width=\textwidth]{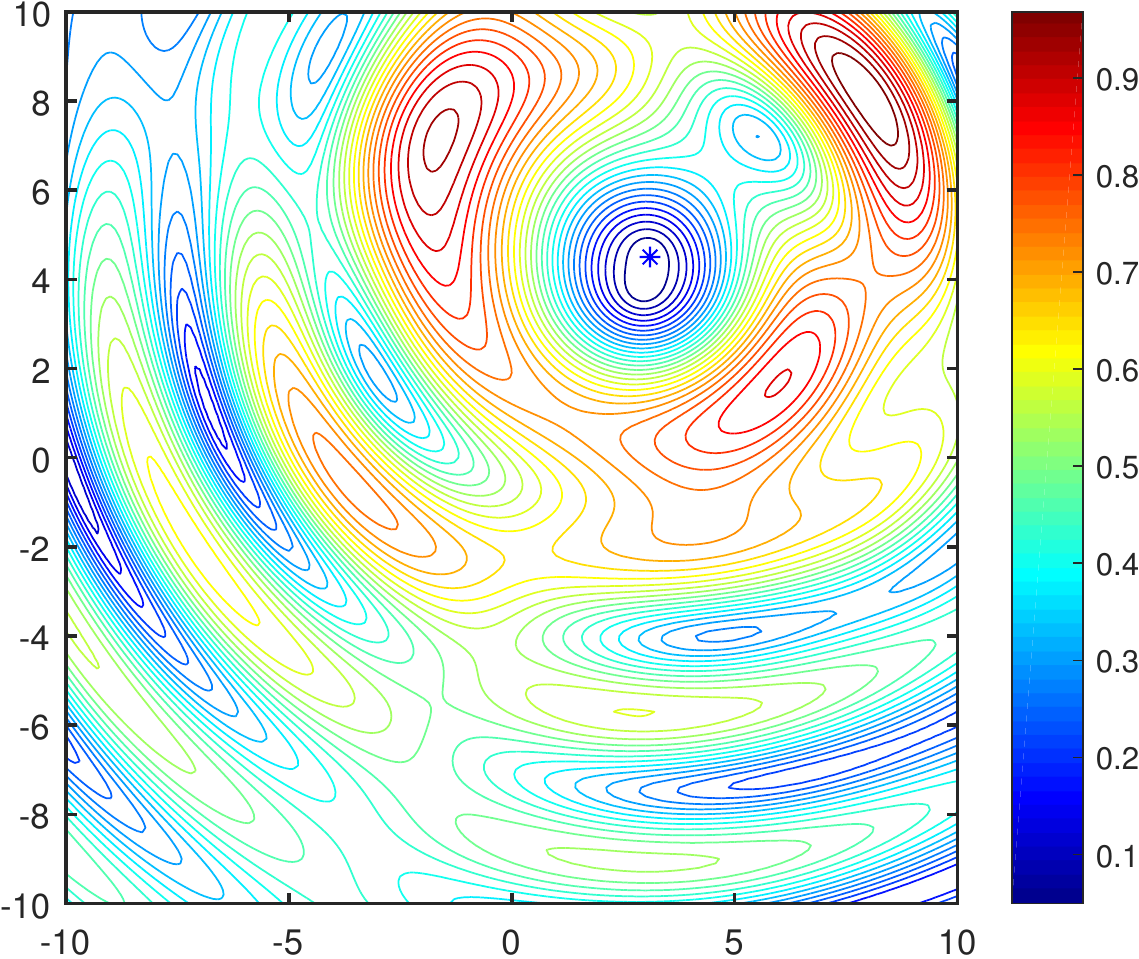}
\vspace{-1cm}
\begin{center}
(c)
\end{center}
\end{minipage}
\begin{minipage}[t]{0.40\linewidth}
\includegraphics[angle=0, width=\textwidth]{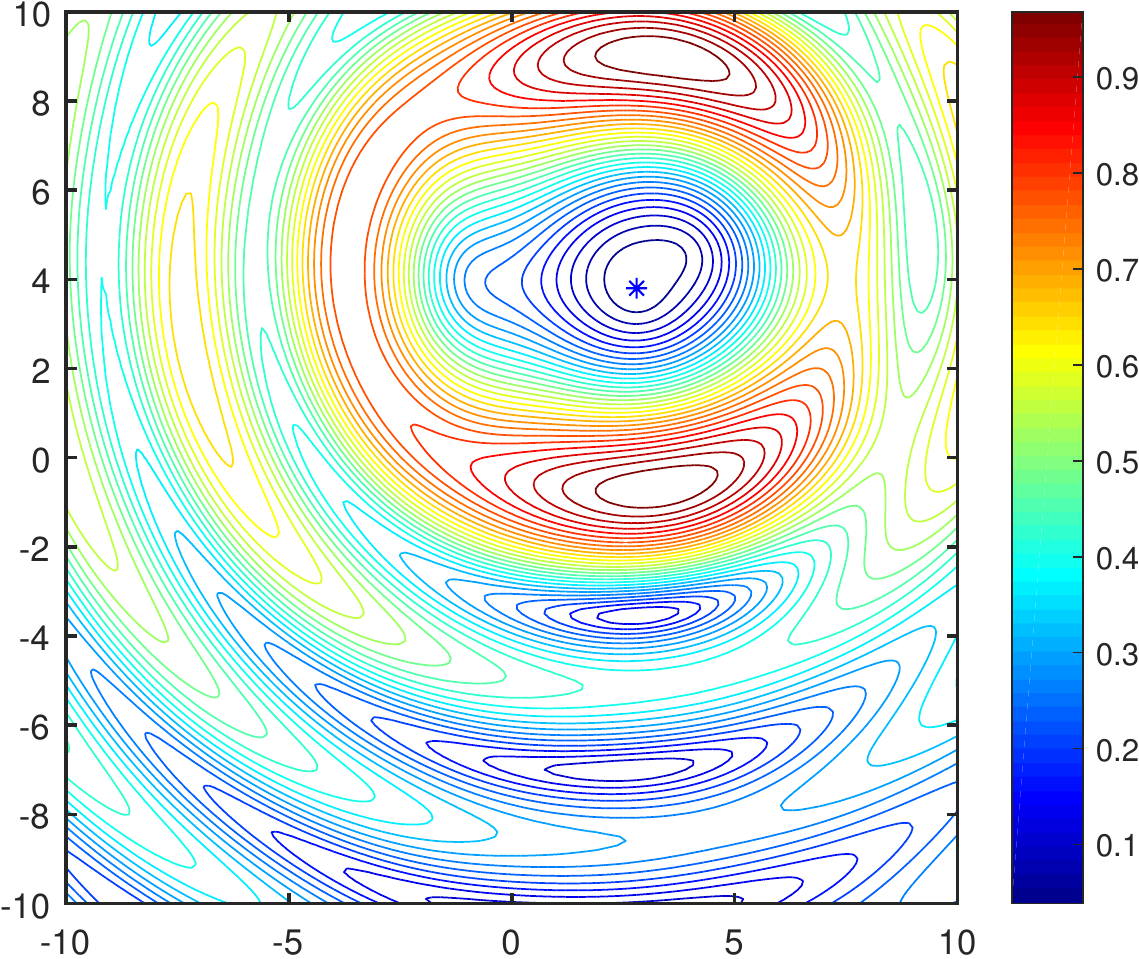}
\vspace{-1cm}
\begin{center}
(d)
\end{center}
\end{minipage}
\caption{\label{fig4}Approximate reconstructions for the kite using sampling discs $R=1$. (a) Dirichlet BC. (b) Neumann BC. (c) Impedance BC.}
\end{figure}


\begin{figure}[h!]
\begin{minipage}[t]{0.30\linewidth}
\includegraphics[angle=0, width=\textwidth]{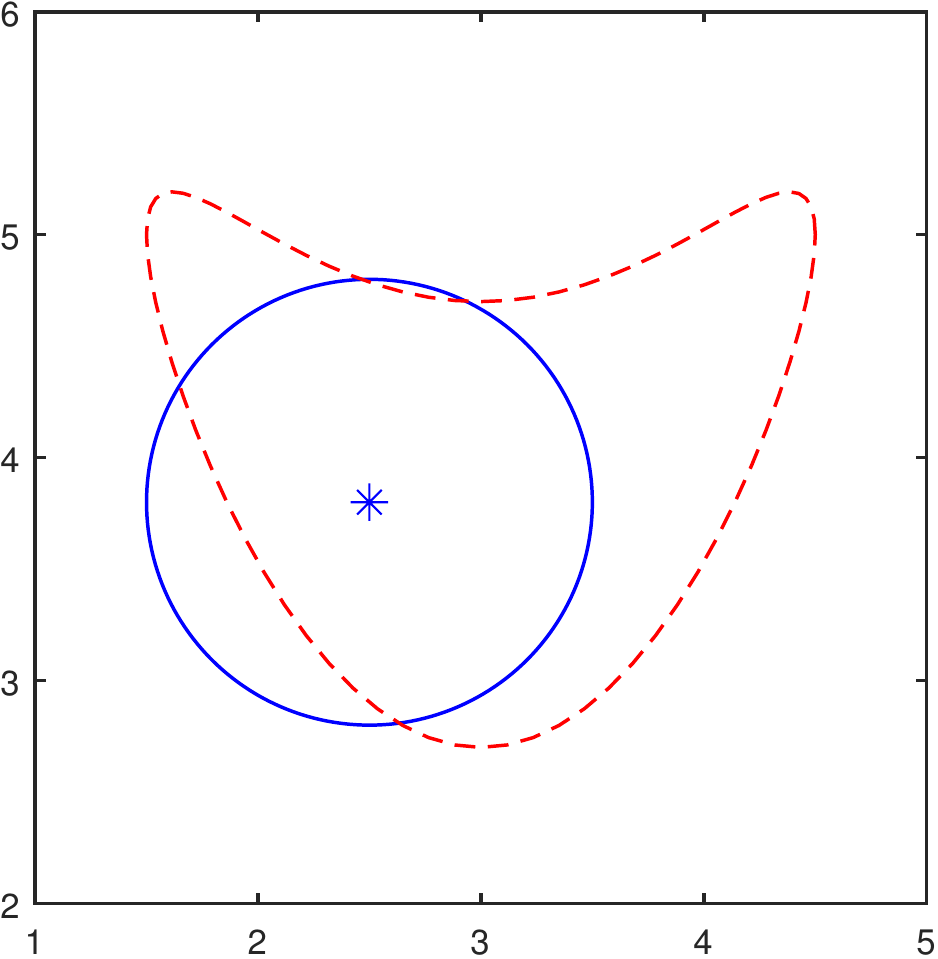}
\vspace{-0.9cm}
\begin{center}
(a)
\end{center}
\end{minipage}
\hspace{0.2cm}
\begin{minipage}[t]{0.30\linewidth}
\includegraphics[angle=0, width=\textwidth]{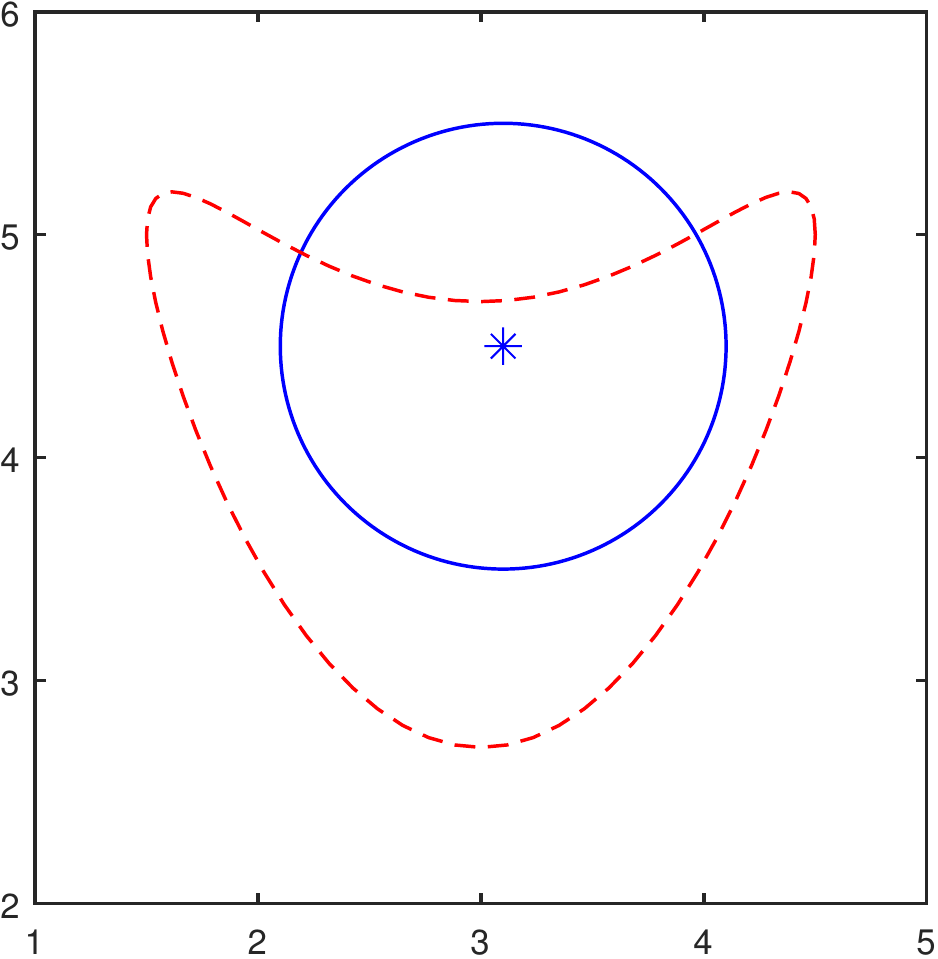}
\vspace{-0.9cm}
\begin{center}
(b)
\end{center}
\end{minipage}
\hspace{0.2cm}
\begin{minipage}[t]{0.30\linewidth}
\includegraphics[angle=0, width=\textwidth]{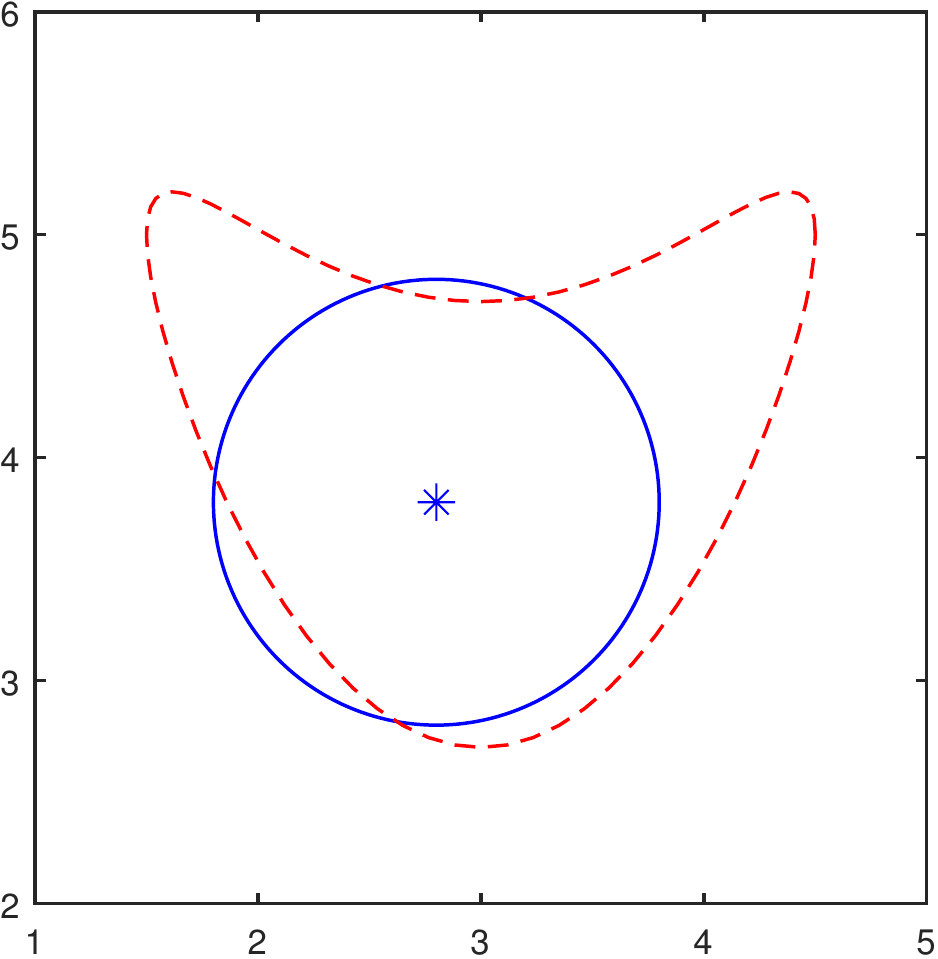}
\vspace{-0.9cm}
\begin{center}
(c)
\end{center}
\end{minipage}
\caption{\label{fig6}Reconstructions of the kite by the ESM using the sampling discs of $R=1$. (a) Dirichlet BC. 
(b) Neumann BC. (c) Impedance BC.}
\end{figure}

\begin{figure}[h!]
\begin{minipage}[t]{\linewidth}
\begin{minipage}[t]{0.40\linewidth}
\includegraphics[angle=0, width=\textwidth]{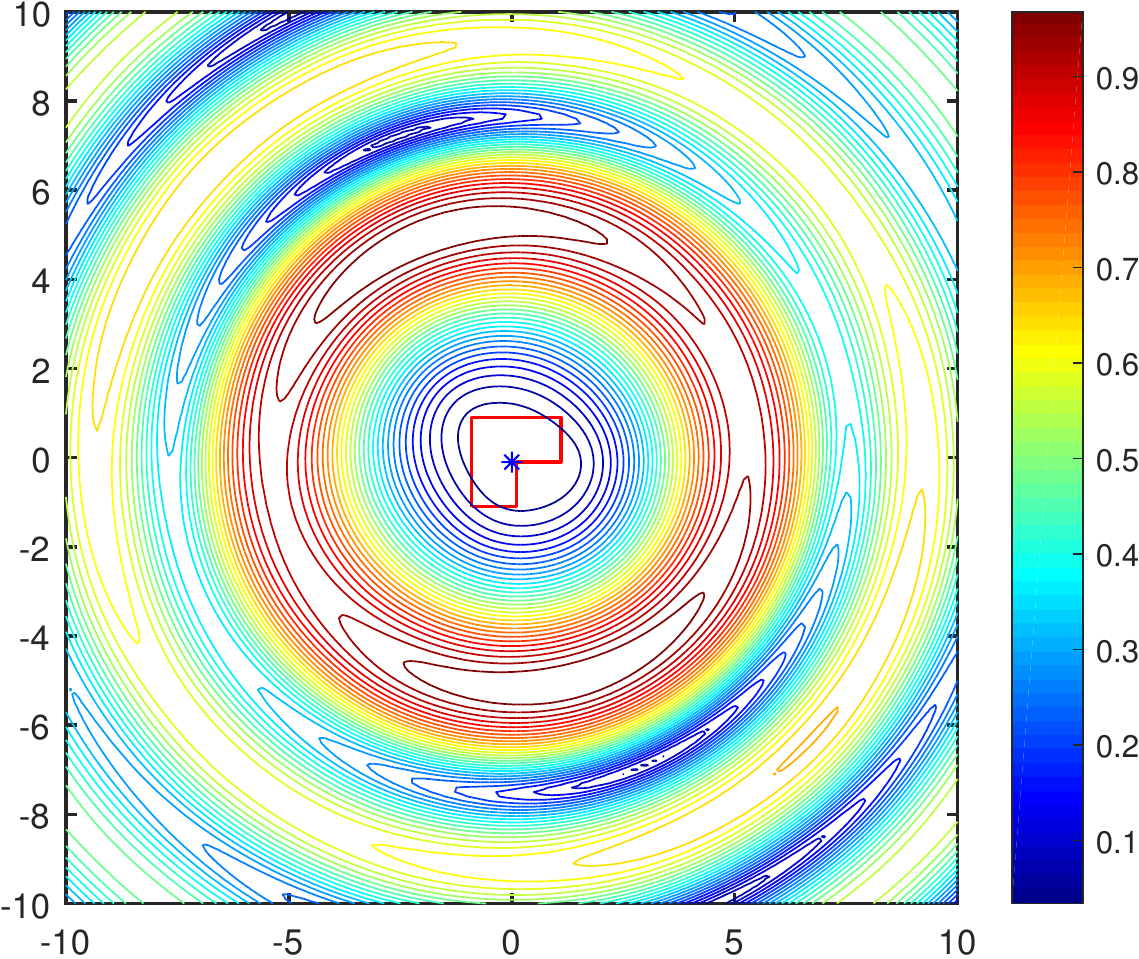}
\vspace{-0.9cm}
\begin{center}
(a)
\end{center}
\end{minipage}
\hspace{0.4cm}
\begin{minipage}[t]{0.40\linewidth}
\includegraphics[angle=0, width=0.83\textwidth]{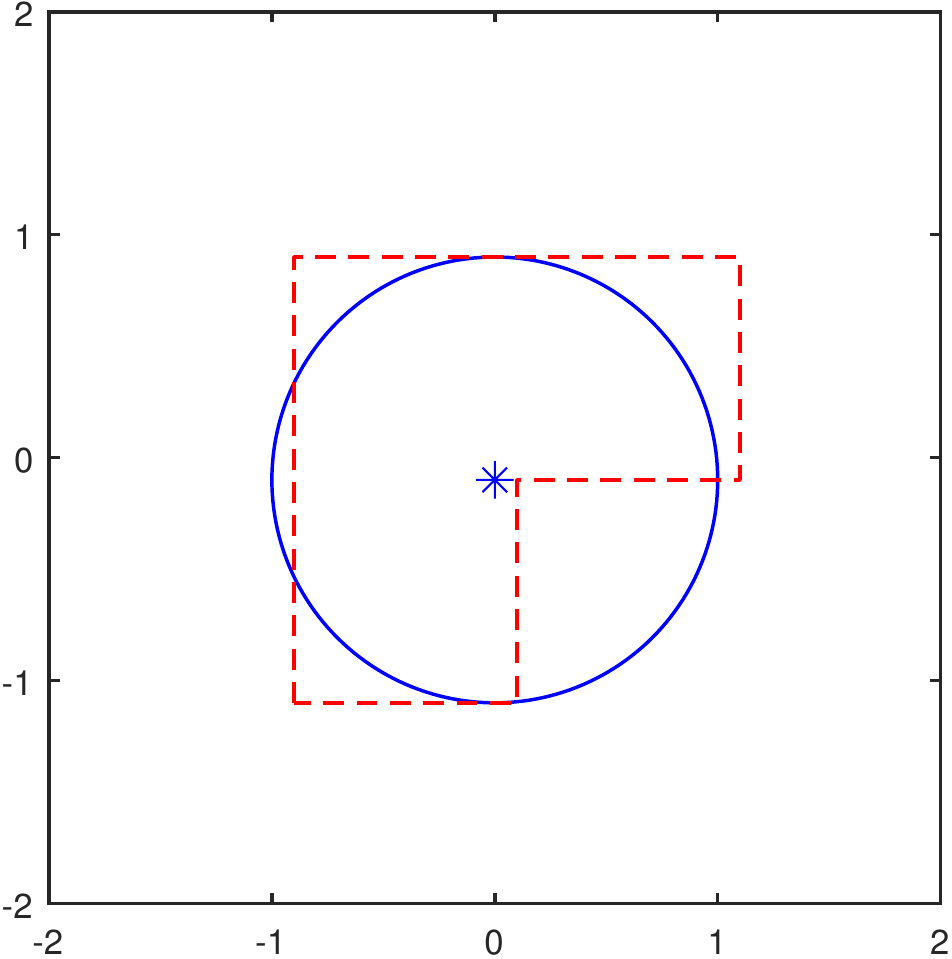}
\vspace{-0.4cm}
\begin{center}
(b)
\end{center}
\end{minipage}
\end{minipage}
\caption{\label{fig7} Reconstruction for L-shaped medium (a) Contours of $I_z$. (b) Reconstruction: blue solid lines are the reconstructions, red dashed lines are the exact obstacle, asterisks are the minimum points of $I_z$.}
\end{figure}

\subsection{Examples for the Multilevel ESM}
Since the size of the scatterer is not known, one needs to decide the proper radius $R$ of the sampling discs. We use the multilevel ESM with a large radius of the sampling discs being $R=2.4$. Then we decrease the radius until a suitable radius is found. 

In the following figures, the solid lines are the reconstructions, red dashed lines are the exact scatterer, asterisks are the minimum points of $I_z$.
\autoref{fig8} shows the results of the multilevel ESM for the triangle object with Dirichlet boundary condition, with the radius of the sampling discs reducing in half from 2.4 to 0.3. When the radius of the sampling discs is $R=0.3$, the minimum point is not inside the reconstruction for $R=0.6$. Hence the proper radius of the sampling discs should be $R=0.6$. Consequently, we obtain the estimation of the size of the target. 

Using sampling discs with radius $0.6$, \autoref{fig9}(a) shows the contours of $I_z$ on the mesh $(-10+0.1m, -10+0.1n), 0\leq m, n\leq 201$, and \autoref{fig9}(b) shows the reconstruction result of the multilevel ESM for the triangle object with Dirichlet boundary condition. .

Using the multilevel ESM, the proper radius $R$ is $0.6$ for the triangle and kite with different boundary conditions, and the proper radius $R$ is computed to be $0.3$ for the L-shaped medium scatterer. \autoref{fig10} shows the reconstructions of the multilevel ESM for the triangle with Neumann and impedance boundary condition ($\lambda=2$), the kite object with Dirichlet, Neumann and impedance boundary condition ($\lambda=2$), and the L-shape medium scatterer.  

\begin{figure}[h!]
\begin{minipage}[t]{\linewidth}
\begin{minipage}[t]{0.40\linewidth}
\includegraphics[angle=0, width=\textwidth]{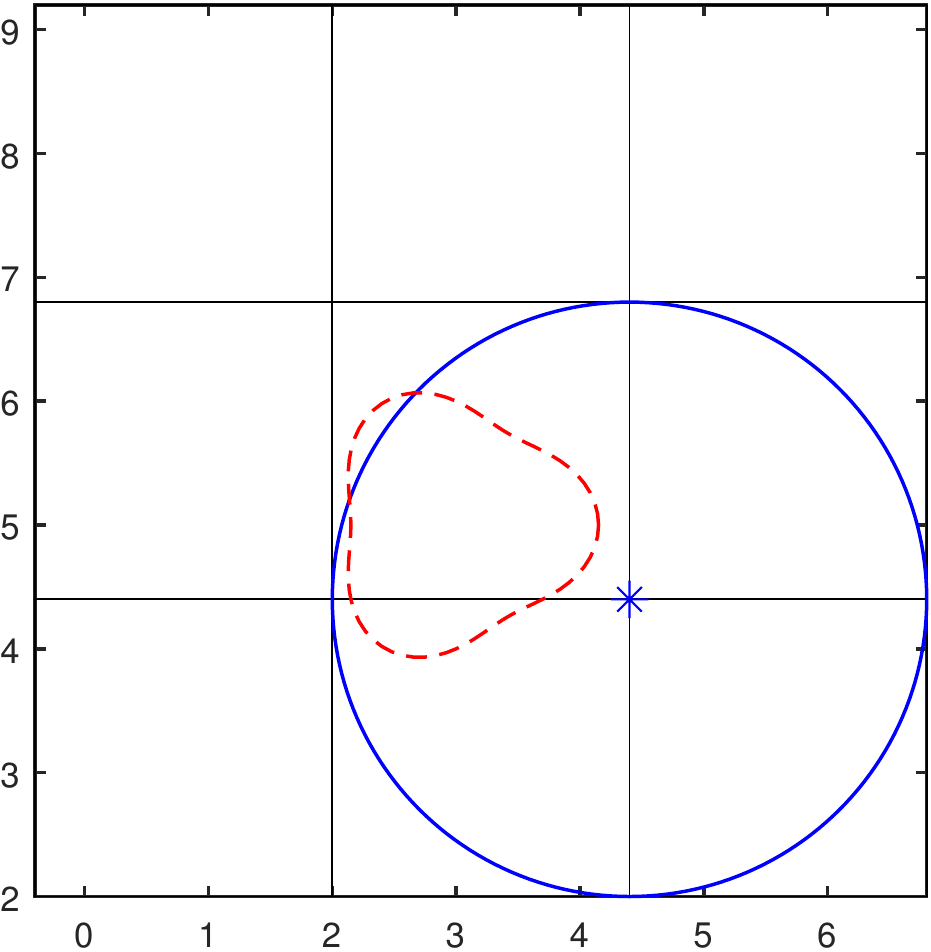}
\vspace{-0.9cm}
\begin{center}
(a)
\end{center}
\vspace{0.1cm}
\end{minipage}
\hspace{0.5cm}
\begin{minipage}[t]{0.40\linewidth}
\includegraphics[angle=0, width=\textwidth]{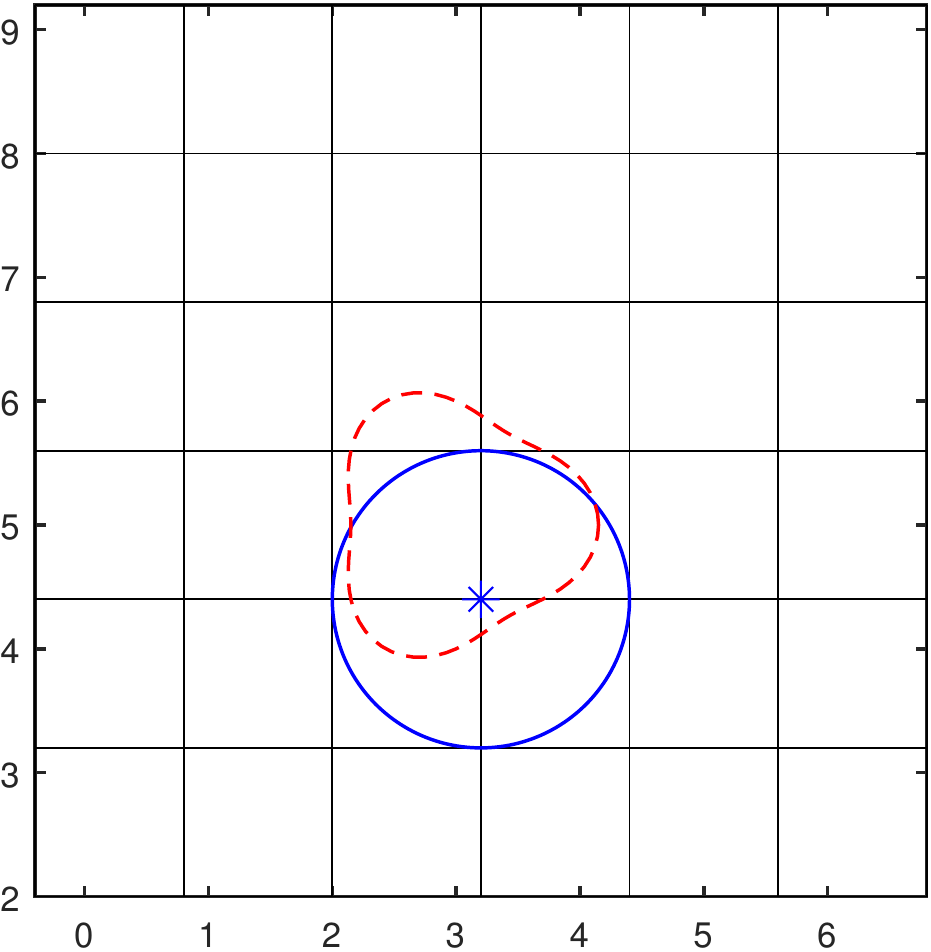}
\vspace{-0.9cm}
\begin{center}
(b)
\end{center}
\vspace{0.1cm}
\end{minipage}
\end{minipage}
\begin{minipage}[t]{\linewidth}
\begin{minipage}[t]{0.40\linewidth}
\includegraphics[angle=0, width=\textwidth]{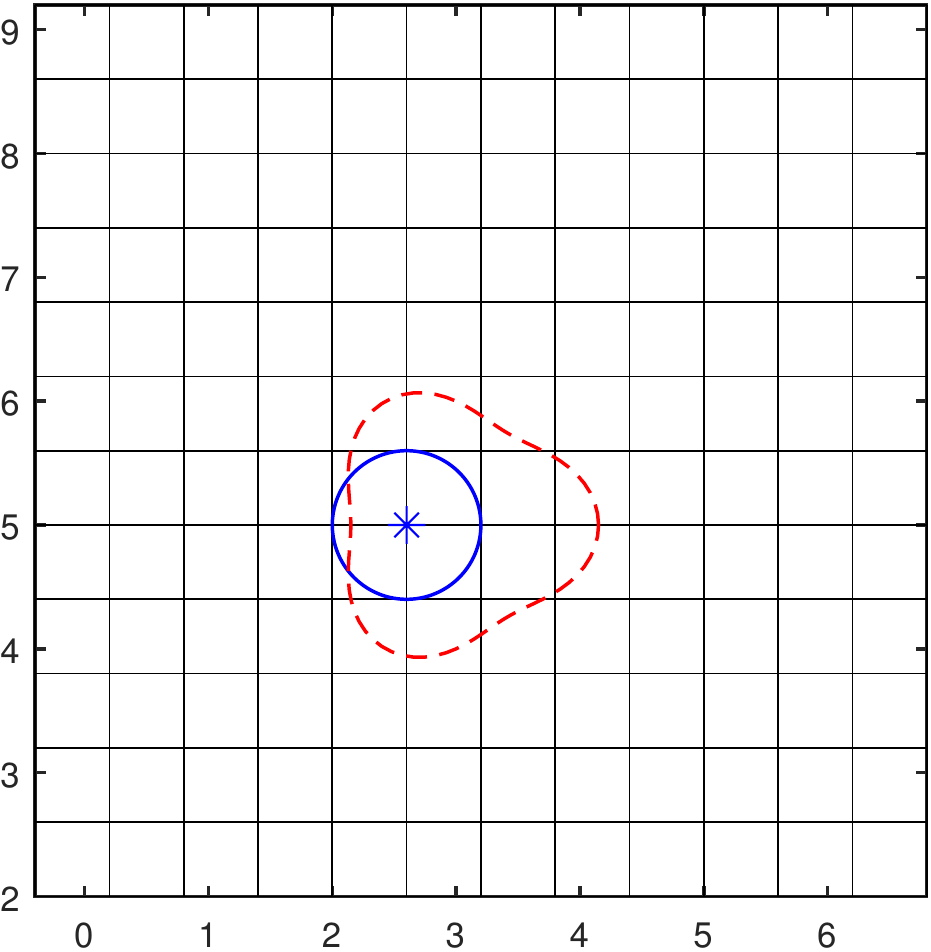}
\vspace{-0.9cm}
\begin{center}
(c)
\end{center}
\end{minipage}
\hspace{0.5cm}
\begin{minipage}[t]{0.40\linewidth}
\includegraphics[angle=0, width=\textwidth]{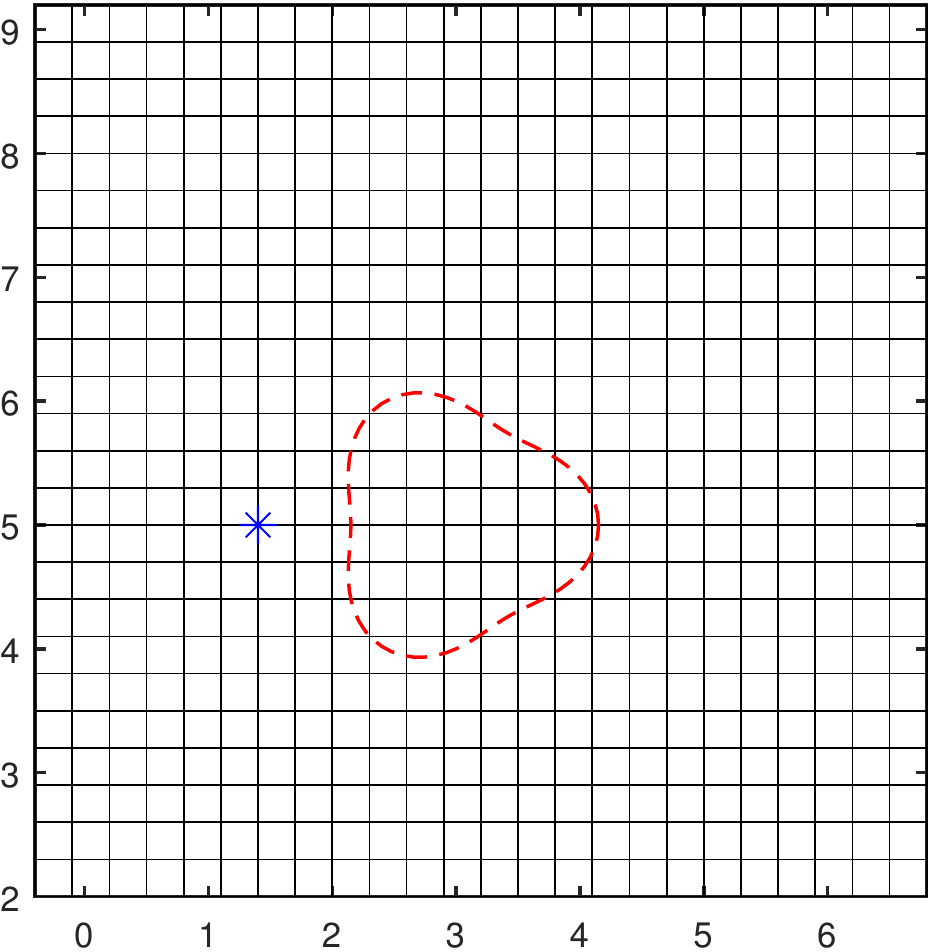}
\vspace{-0.9cm}
\begin{center}
(d)
\end{center}
\end{minipage}
\end{minipage}
\caption{\label{fig8}The multilevel ESM for the triangle with Dirichlet BC: (a) The radius $R$ of the sampling discs is $R=2.4$. (b) $R=1.2$. (c) $R=0.6$. (d) $R=0.3$}
\end{figure}

\begin{figure}[h!]
\begin{minipage}[t]{\linewidth}
\begin{minipage}[t]{0.40\linewidth}
\includegraphics[angle=0, width=\textwidth]{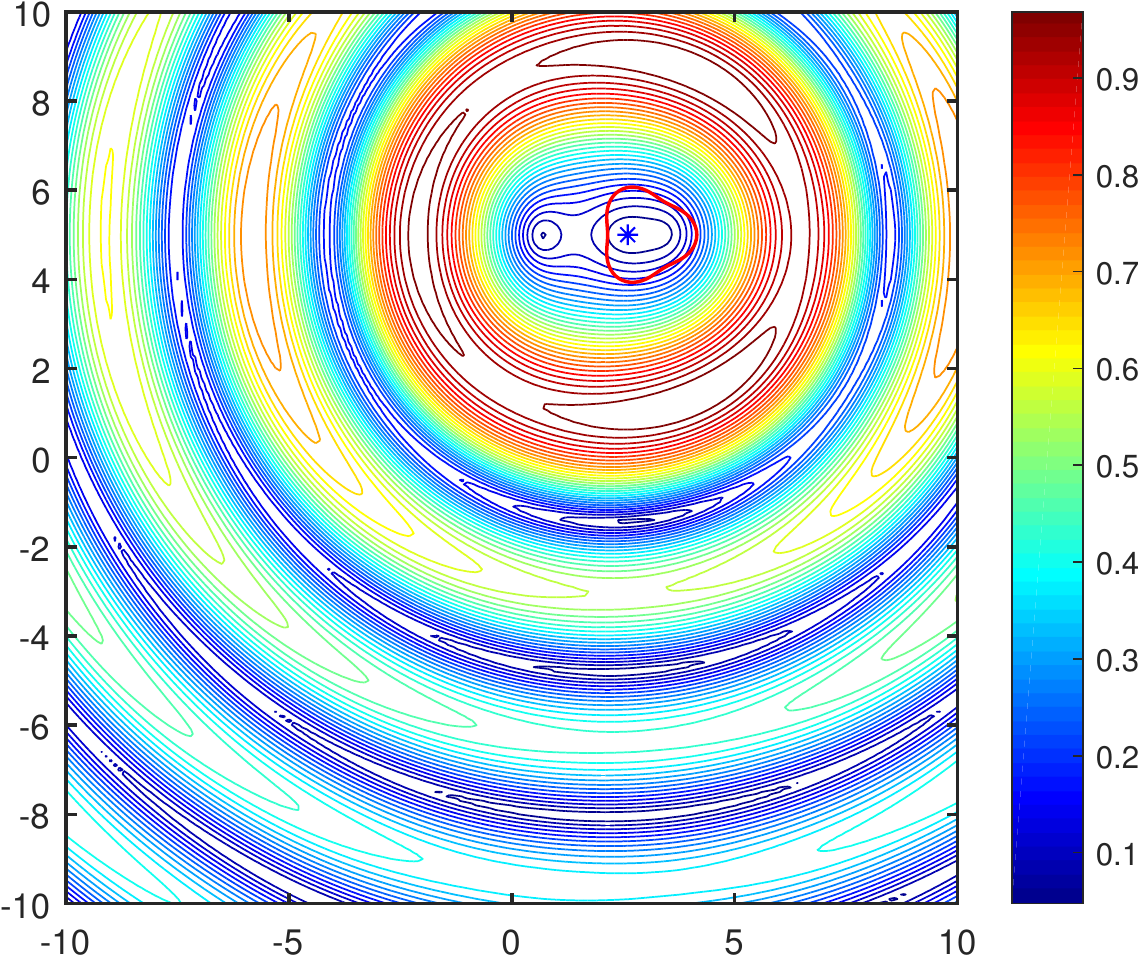}
\vspace{-0.9cm}
\begin{center}
(a)
\end{center}
\end{minipage}
\hspace{0.4cm}
\begin{minipage}[t]{0.40\linewidth}
\includegraphics[angle=0, width=0.83\textwidth]{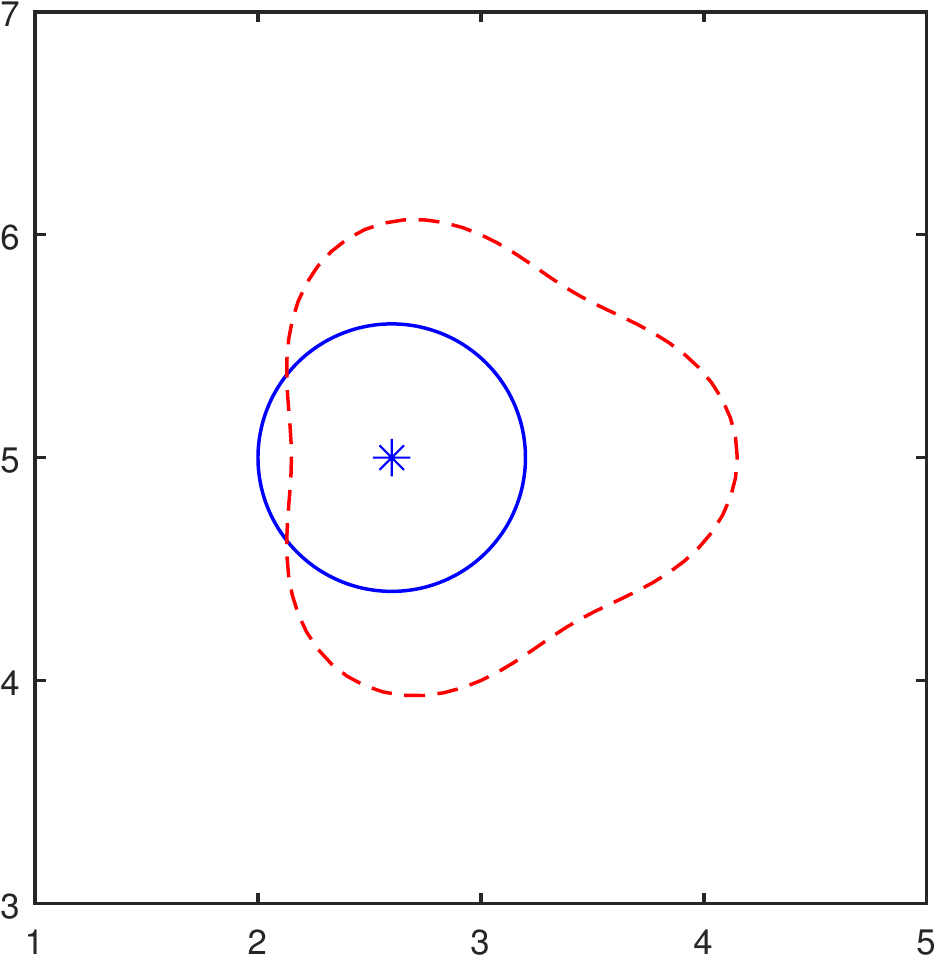}
\vspace{-0.4cm}
\begin{center}
(b)
\end{center}
\end{minipage}
\end{minipage}
\caption{\label{fig9} The multilevel ESM for the triangle with Dirichlet BC. (a) Contour plot. (b) Reconstruction.}
\end{figure}

\begin{figure}[h!]
\begin{minipage}[t]{\linewidth}
\begin{minipage}[t]{0.30\linewidth}
\includegraphics[angle=0, width=\textwidth]{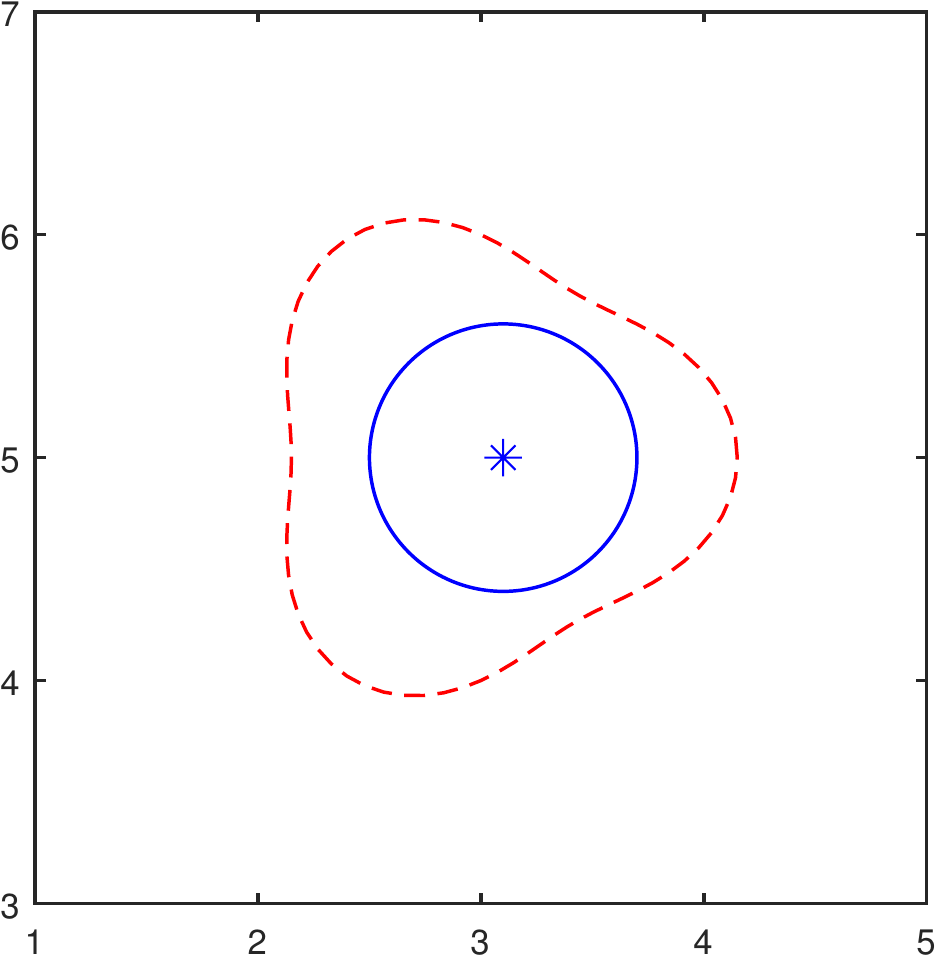}
\vspace{-0.9cm}
\begin{center}
(a)
\end{center}
\vspace{0.1cm}
\end{minipage}
\hspace{0.2cm}
\begin{minipage}[t]{0.30\linewidth}
\includegraphics[angle=0, width=\textwidth]{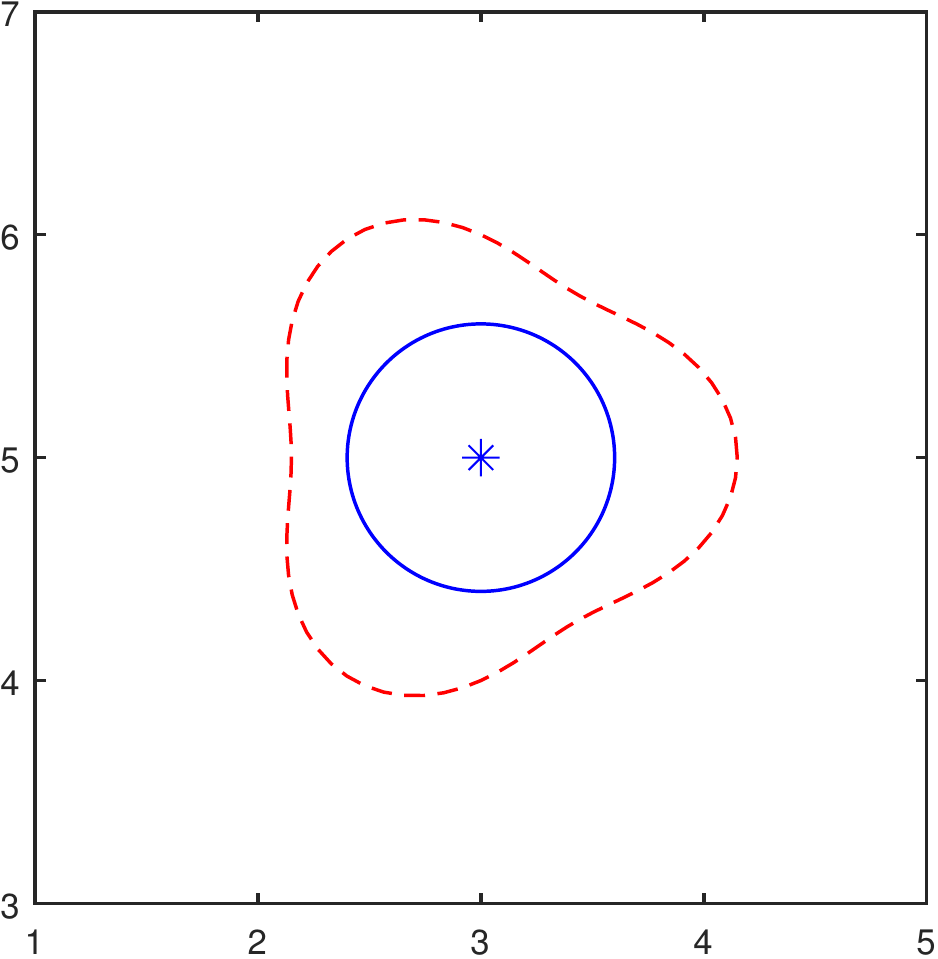}
\vspace{-0.9cm}
\begin{center}
(b)
\end{center}
\vspace{0.1cm}
\end{minipage}
\hspace{0.2cm}
\begin{minipage}[t]{0.30\linewidth}
\includegraphics[angle=0, width=\textwidth]{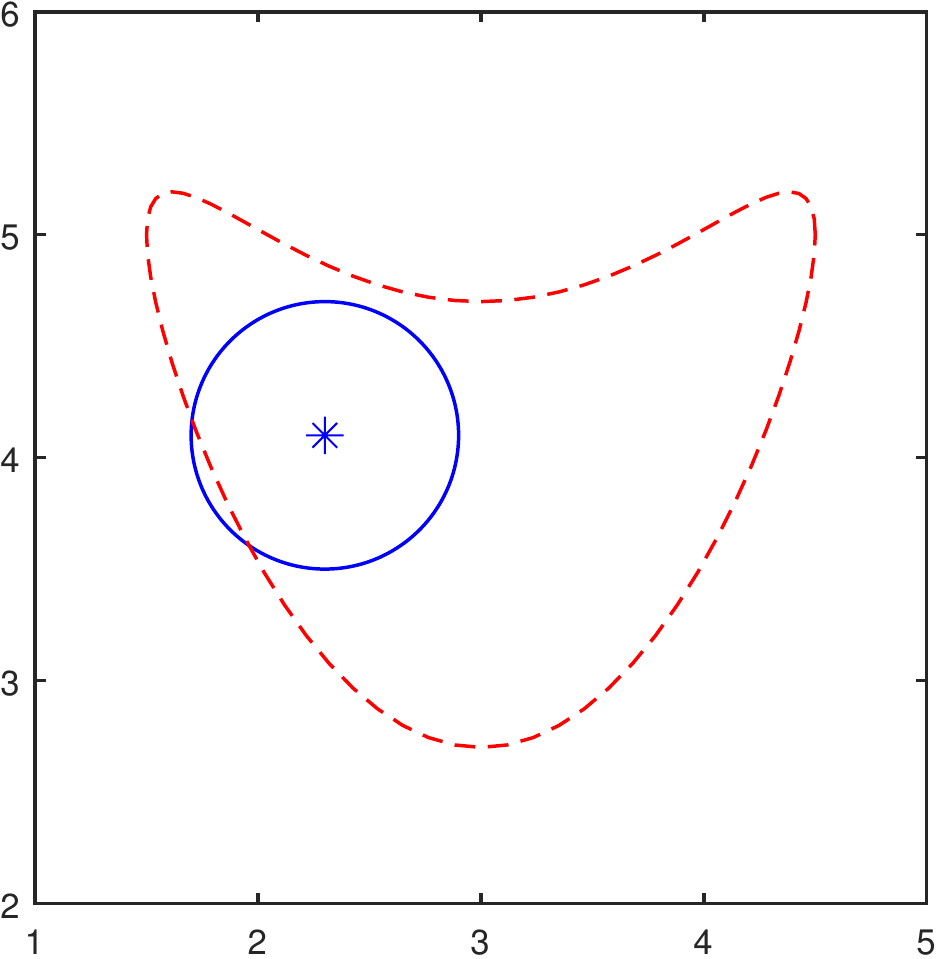}
\vspace{-0.9cm}
\begin{center}
(c)
\end{center}
\vspace{0.1cm}
\end{minipage}
\end{minipage}

\begin{minipage}[t]{\linewidth}
\begin{minipage}[t]{0.30\linewidth}
\includegraphics[angle=0, width=\textwidth]{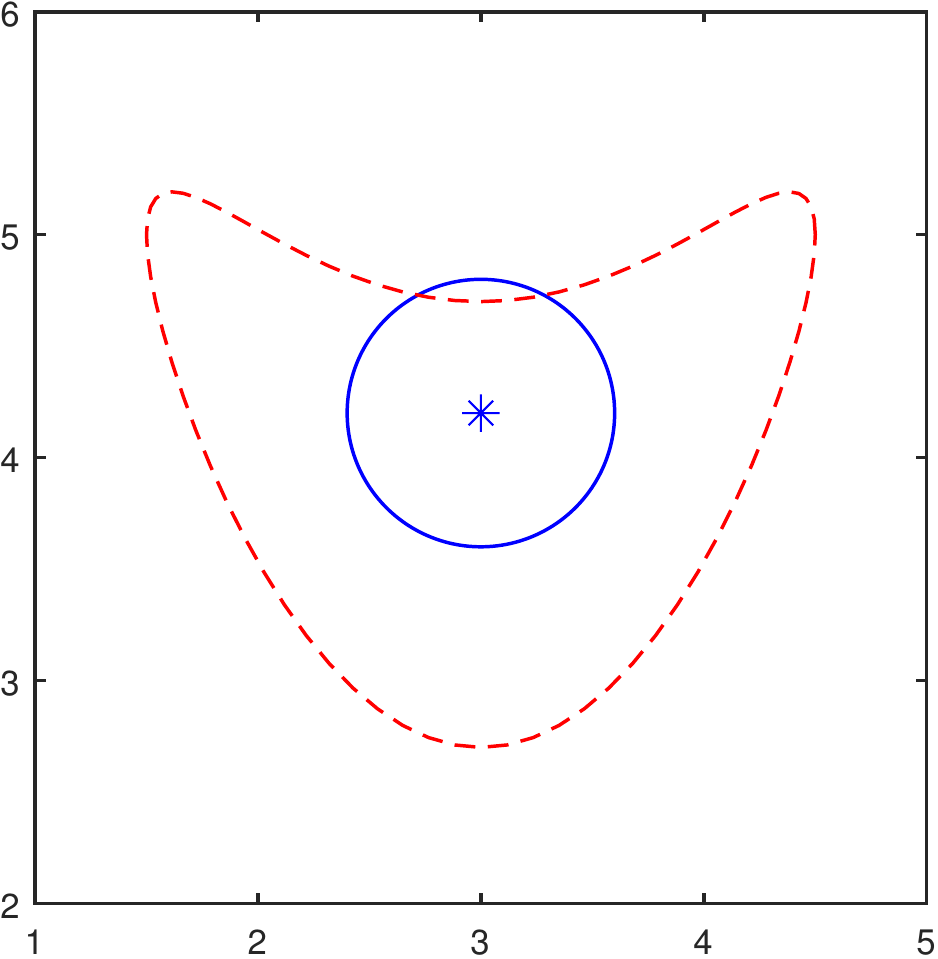}
\vspace{-0.9cm}
\begin{center}
(d)
\end{center}
\end{minipage}
\hspace{0.2cm}
\begin{minipage}[t]{0.30\linewidth}
\includegraphics[angle=0, width=\textwidth]{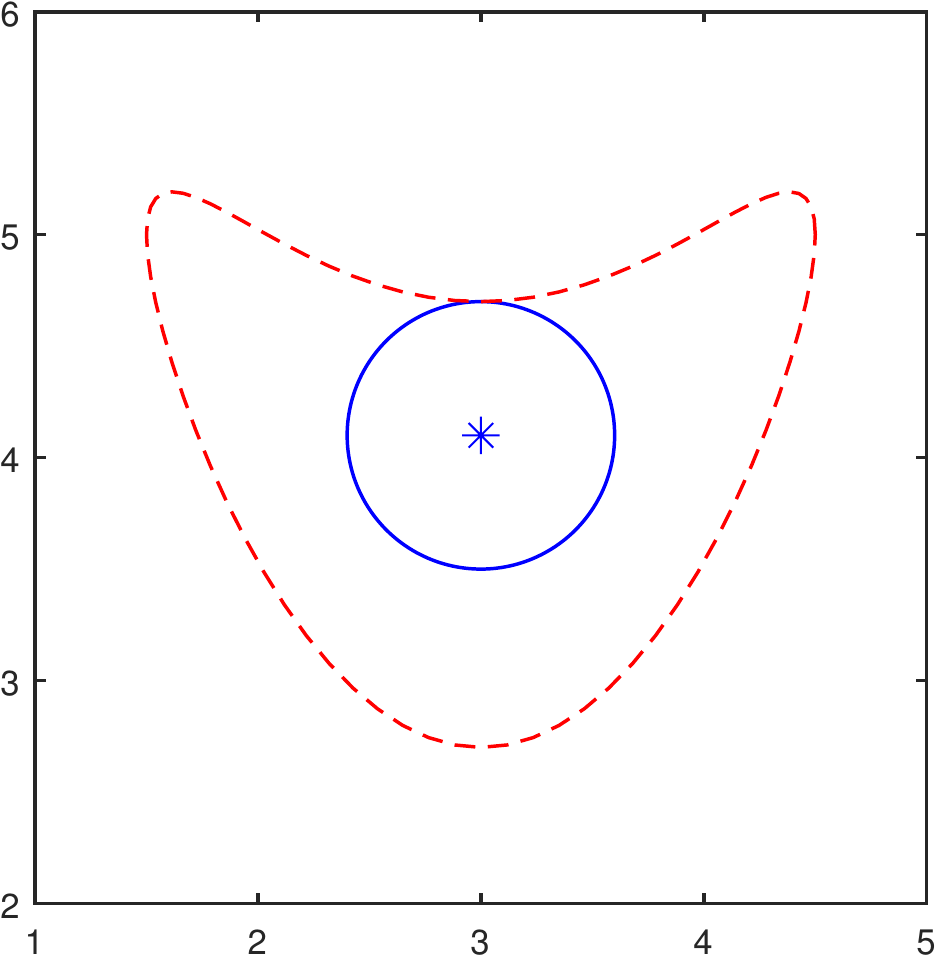}
\vspace{-0.9cm}
\begin{center}
(e)
\end{center}
\end{minipage}
\hspace{0.2cm}
\begin{minipage}[t]{0.30\linewidth}
\includegraphics[angle=0, width=\textwidth]{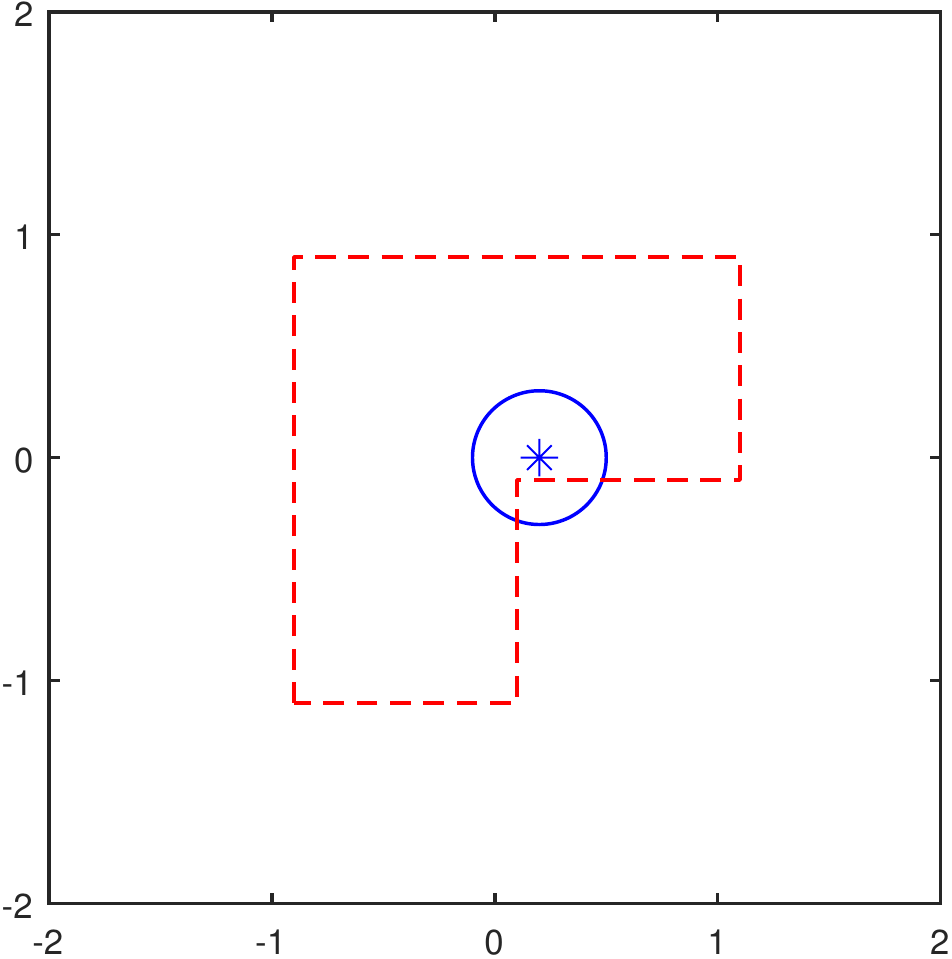}
\vspace{-0.9cm}
\begin{center}
(f)
\end{center}
\end{minipage}
\end{minipage}
\caption{\label{fig10} The reconstruction results of the multilevel ESM for (a) the triangle obstacle with Neumann BC; (b) the triangle obstacle with impedance BC; (c) the kite obstacle with Dirichlet BC; (d) the kite obstacle with Neumann BC; (e) the kite obstacle with impedance BC; and (f) the L-shaped inhomogeneous medium.}
\end{figure}

\section{Conclusion}\label{CF}
A new sampling method using far field data of one incident wave is proposed.
It can correctly obtain the location and approximate support of the scatter and is wavenumber independent.
In this paper, the sampling discs are sound-soft. One can also use sound-hard discs or discs with other boundary conditions. 
Furthermore, since the kernel of the far field equation is computed, the ESM can process various input data. 
In future, we plan to study the following variations of the ESM.

\begin{itemize}
\item{Limited Aperture Observation Data}

Let $\mathds{S}_0$ be a non-trivial proper subset of $\mathds{S}$. By analyticity, $u_\infty(\hat{x})$ for all observation directions $\hat{x}\in \mathds{S}$  
can be uniquely determined by $u_\infty(\hat{x})$ for limited observation directions $\hat{x}\in \mathds{S}_0$. It is still possible to employ the ESM
by solving the far field equations of limited observation data:
\begin{equation}\label{feS0}
\mathcal{F}_z g(\hat{x})=u^\infty(\hat{x}),\ \ \ \ \hat{x}\in \mathds{S}_0.
\end{equation}
Then the indicator for the sampling point $z$ is taken as
\begin{equation}\label{indicator}
I_z = \|g_z^\epsilon\|_{L^2}, \quad z \in T,
\end{equation}
where $g_z$ is the regularized solution of \eqref{feS0} for the partial far field data $u_\infty(\hat{x}), \hat{x}\in \mathds{S}_0$.

\item {Multiple Incident Direction Data}

For data for multiple incident directions $u_\infty(\hat{x}, d_i), \hat{x}\in \mathds{S}_0, i=1, \ldots, M$, one can solve the
far field equation \eqref{feS0} for each incident direction $d_i$. The indicator for a sampling point $z$ can be set as
\begin{equation}\label{indicatordi}
I_z = \sum_{i=1}^M I_z^{d_i} = \sum_{i=1}^M\|g_z^\epsilon(d_i)\|_{L^2}, \quad z \in T,
\end{equation}
where $g_z^{d_i}$ is the regularized solution of \eqref{feS0} for $u_\infty(\hat{x}, d_i), \hat{x}\in \mathds{S}_0$.

\item {Multiple Frequency Data}
In a similar way, the ESM can process multiple frequency data
\[
u_\infty(\hat{x}, d_i, k_j),\quad \hat{x}\in \mathds{S}_0, i=1, \ldots, M, j=1, \ldots, J.
\]
Again, one solves the far field equation \eqref{feS0} for each incident direction $d_i$ and each frequency $k_j$. 
The indicator for a sampling point $z$ is then
\begin{equation}\label{indicatordikj}
I_z = \sum_{j=1}^J \sum_{i=1}^M I_z^{d_i, k_j} = \sum_{j=1}^J \sum_{i=1}^M\|g_z^\epsilon(d_i, k_j)\|_{L^2}, \quad z \in T,
\end{equation}
where $g_z^{d_i, k_j}$ is the regularized solution of \eqref{feS0} for $u_\infty(\hat{x}, d_i, k_j), \hat{x}\in \mathds{S}_0$.

\end{itemize}

\section*{Acknowledgments}
The research of J. Liu was supported in part by the Guangdong Natural Science Foundation (No. 2016A030313074) and by
a research fund from Jinan University (No. 21617414). The research of J. Sun was partially supported by NSFC (No. 11771068)



\end{document}